\documentclass[12pt, reqno]{amsart}

\usepackage{amsmath, amsthm, amssymb, stmaryrd}
\usepackage{hyperref}
\usepackage{enumerate}
\usepackage{url}

\usepackage{bm}

\usepackage{xcolor}

\usepackage[centering, margin={1in, 0.5in}, includeheadfoot]{geometry}

\usepackage{fancyvrb}

\newtheorem{thm}{Theorem}[section]
\newtheorem{lemma}[thm]{Lemma}
\newtheorem{prop}[thm]{Proposition}

\theoremstyle{definition}
\newtheorem{defn}[thm]{Definition}

\theoremstyle{remark}
\newtheorem{remark}[thm]{Remark}

\numberwithin{equation}{section}

\DeclareMathOperator{\Mod}{mod}
\newcommand{\mmod}[1]{\;(\Mod{ #1})}

\def\alp{{\alpha}} 
\def\bet{{\beta}}  
\def\gam{{\gamma}} 
\def\del{{\delta}} \def\Del{{\Delta}}

\def\kap{{\kappa}}
\def\lam{{\lambda}} \def\Lam{{\Lambda}}

\def\sig{{\sigma}}

\def\ome{{\omega}}  
\def\eps{\varepsilon}

\def\le{\leqslant} \def\ge{\geqslant}

\def\d{{\,{\rm d}}}

\def \sig{{\sigma}}

\def \bC {\mathbb C}

\def \bF {\mathbb F}

\def \bN {\mathbb N}

\def \bQ {\mathbb Q}
\def \bR {\mathbb R}
\def \bZ {\mathbb Z}
\def \bT {\mathbb T}

\def \bn {{\mathbf{n}}}

\def \bu {\mathbf u}
\def \bv {\mathbf v}
\def \bx {\mathbf x}

\def \by {\mathbf y}
\def \bz {\mathbf z}

\def \balp {\boldsymbol{\alp}}

\def \bbeta {\boldsymbol{\beta}}

\def \fm {\mathfrak m}

\def \fM {\mathfrak M}

\def \cA {\mathcal A}
\def \cB {\mathcal B}
\def \cC {\mathcal C}

\def \cF {\mathcal F}

\def \cP {\mathcal P}

\def \cS {\mathcal S}

\def \ord {\mathrm{ord}}

\def \deg {\mathrm{deg}}

\def \N {\mathbb{N}}

\def \R {\mathbb{R}}
\def \T {\mathbb{T}}
\def \Z {\mathbb{Z}}

\def \balpha {\bm{\alpha}}
\def \btheta {\bm{\theta}}

\newcommand{\str}{\mathrm{str}}
\newcommand{\sml}{\mathrm{sml}}
\newcommand{\unf}{\mathrm{unf}}
\newcommand{\bohr}{\mathrm{Bohr}}

\begin{document}
\title[Generalised Rado and Roth criteria]{Generalised Rado and Roth criteria}
\subjclass[2020]{11B30 (primary); 05D10, 11D72, 11L15 (secondary)}
\keywords{Arithmetic combinatorics, arithmetic Ramsey theory, Diophantine equations, Hardy--Littlewood method, partition regularity, restriction theory}
\author{Jonathan Chapman \and Sam Chow}

\address{School of Mathematics, Fry Building, University of Bristol, Woodland Road, Bristol, BS8 1UG, United Kingdom}
\email{jonathan.chapman@bristol.ac.uk}
\address{Mathematics Institute, Zeeman Building, University of Warwick, Coventry CV4 7AL, United Kingdom}
\email{Sam.Chow@warwick.ac.uk}

\maketitle

{\centering\footnotesize \em{Dedicated to Sean Prendiville} \par}

\begin{abstract} 
We study the Ramsey properties of equations
$a_1P(x_1) + \cdots + a_sP(x_s) = b$, where $a_1,\ldots,a_s,b$ are integers, and $P$ is an integer polynomial of degree $d$. Provided there are at least $(1+o(1))d^2$ variables, we show that Rado's criterion and an intersectivity condition completely characterise which equations of this form admit monochromatic solutions with respect to an arbitrary finite colouring of the positive integers. Furthermore, we obtain a Roth-type theorem for these equations, showing that they admit non-constant solutions over any set of integers with positive upper density if and only if $b=a_1+\cdots+a_s = 0$. In addition, we establish sharp asymptotic lower bounds for the number of monochromatic/dense solutions (supersaturation).
\end{abstract}

\setcounter{tocdepth}{1}
\tableofcontents

\section{Introduction}

A system of polynomial equations is called \emph{partition regular} if every finite colouring of the positive integers admits monochromatic non-constant\footnote{A solution $(x_{1},\ldots,x_{s})$ is \emph{non-constant} if $x_i\neq x_j$ holds for some $i\neq j$.} solutions to the system.\footnote{Some authors allow constant monochromatic solutions in the definition of partition regularity.} 
A foundational result in the field of arithmetic Ramsey theory is \emph{Rado's criterion} \cite[Satz IV]{Rad1933}, which provides necessary and sufficient conditions for a finite system of linear equations to be partition regular. For example, given $s\geqslant 3$ and non-zero integers $a_{1},\ldots,a_s$, Rado's criterion asserts that the linear homogeneous equation
\begin{equation}\label{eqn1.1}
    a_{1}x_{1} + \cdots + a_s x_s = 0
\end{equation}
is partition regular if and only if there exists a non-empty set $I\subseteq\{1,\ldots,s\}$ such that $\sum_{i\in I}a_i = 0$.

A similar, stronger notion is that of \emph{density regularity}, which refers to systems of equations which have non-constant solutions over all sets of positive integers $A$ satisfying
\begin{equation*}
    \limsup_{N\to\infty}\frac{|A\cap\{1,2,\ldots,N\}|}{N}>0.
\end{equation*}
Such sets $A$ are said to have \emph{positive upper density}. An influential Fourier analytic argument of Roth \cite{Rot1954} shows that if $s\geqslant 3$, then the linear homogeneous equation (\ref{eqn1.1}) is density regular if and only if $a_1 + \cdots + a_s = 0$.

Recent work on partition regularity has focused on generalising the theorems of Rado and Roth by finding necessary \cite{BaLuMo21,DL18} and sufficient  \cite{Cha2022,Cho2018,CLP2021,Pre2021,Sch2021} conditions for partition and density regularity for general systems of polynomial equations. In this paper we consider equations of the form
\begin{equation}\label{eqn1.2}
    a_{1}P(x_{1}) + \cdots + a_{s}P(x_{s}) = 0,
\end{equation}
where $P$ is a polynomial with integer coefficients, and $a_1,\ldots,a_s$ are non-zero integers. Previous work of the second author with Lindqvist and Prendiville \cite{CLP2021} extended Rado's criterion to equations (\ref{eqn1.2}) for $P(x)=x^{d}$ under the assumption that the number of variables $s$ is sufficiently large in terms of $d$.

In this paper, we extend these results further by completely characterising partition and density regularity for equations (\ref{eqn1.2}) in sufficiently many variables. To state our main results, we require the following definition.
An integer polynomial $P\in\Z[X]$ is called \emph{intersective} if for every positive integer $n$, there exists an integer $x$ such that $P(x)$ is divisible by $n$. Integer polynomials which admit integer zeros are intersective, however, there exist numerous intersective polynomials which have no rational zeros, such as $P(X)=(X^{3} -19)(X^{2} +X +1)$.

Our first theorem shows that Rado's criterion and Roth's theorem hold for equations in intersective polynomials with sufficiently many variables.

\begin{thm}\label{thm1.1}
Let $d\geqslant 2$ be an integer, and define
\begin{equation}\label{eqn1.3}
s_1(d) := \begin{cases}
5, &\text{if } d=2; \\
9, &\text{if } d=3; \\
d^2-d+2\lfloor \sqrt{2d+2} \rfloor + 1, &\text{if } d \ge 4.
\end{cases}
\end{equation}
Let $P$ be an intersective integer polynomial of degree $d$. Let $s\geqslant s_1(d)$ be an integer, and let $a_1,\ldots,a_s$ be non-zero integers. 
\begin{enumerate}
    \item[(PR)] The equation (\ref{eqn1.2}) is partition regular if and only if there exists a non-empty set $I\subseteq\{1,\ldots,s\}$ such that $\sum_{i\in I}a_i =0$.
    \item[(DR)] The equation (\ref{eqn1.2}) is density regular if and only if $a_1 +\cdots +a_s = 0$.
\end{enumerate}
\end{thm}

\begin{remark} As we will soon clarify, intersectivity is also necessary.
\end{remark}

By performing a change of variables, one may interpret Theorem \ref{thm1.1} as generalisations of Rado and Roth's theorems to colourings and dense subsets respectively of the image set $P(\N):=\{P(1),P(2),P(3),\ldots\}$. 
More precisely, if $s\geqslant s_1(d)$ and $\sum_{i\in I}a_i =0$ for some non-empty $I\subseteq\{1,\ldots,s\}$, then Theorem \ref{thm1.1} asserts that the linear equation (\ref{eqn1.1}) admits non-constant monochromatic solutions with respect to any finite colouring of $P(\N)$. Similarly, if $a_1 + \cdots + a_s =0$, then Theorem \ref{thm1.1} implies that (\ref{eqn1.1}) has non-constant solutions over any set of positive integers $A$ satisfying
\begin{equation*}
    \limsup_{N\to\infty}\frac{|A\cap\{P(1),\ldots,P(N)\}|}{N}>0.
\end{equation*}

\subsection{Inhomogeneous equations}

Rado \cite{Rad1933} also studied inhomogeneous linear equations
\begin{equation}\label{eqn1.4}
    a_{1}x_{1} + \cdots + a_{s}x_{s} = b,
\end{equation}
where $a_1,\ldots,a_s$ are non-zero integers and $b$ is a fixed integer. Rado showed that every finite colouring of the positive integers admits (possibly constant) monochromatic solutions to (\ref{eqn1.4}) if and only if $(a_{1} + \cdots + a_{s})$ divides $b$. If one does not permit constant solutions, then it was noted by Hindman and Leader \cite[Theorem 3.4]{HL2006} that (\ref{eqn1.4}) is partition regular if and only if $(a_{1} + \cdots + a_{s})$ divides $b$ and $\sum_{i\in I}a_i =0$ for some non-empty $I\subseteq\{1,\ldots,s\}$. Note that, by considering solutions over a non-zero residue class modulo a sufficiently large prime $p$, equation
(\ref{eqn1.4}) cannot be density regular if $b\neq 0$.

Our second theorem, of which Theorem \ref{thm1.1} is a special case, comprehensively characterises partition and density regularity for arbitrary polynomial analogues  of (\ref{eqn1.4}) in sufficiently many variables.

\begin{thm}\label{thm1.3}
Let $d\geqslant 2$ be an integer, and define $s_1(d)$ by (\ref{eqn1.3}). 
Let $P$ be an integer polynomial of degree $d$, and let $s\geqslant s_1(d)$ be an integer. Let $a_1,\ldots,a_s$ be non-zero integers, and let $b$ be an integer. Consider the equation
\begin{equation}\label{eqn1.5}
    a_{1}P(x_{1}) + \cdots + a_{s}P(x_{s}) = b.
\end{equation}
\begin{enumerate}
    \item[(PR)]  The equation (\ref{eqn1.5}) is partition regular if and only if there exists a non-empty set $I\subseteq\{1,\ldots,s\}$ such that $\sum_{i\in I}a_i =0$ and an integer $m$ with $b=(a_{1}+\cdots +a_s)m$ such that $P(X)-m$ is an intersective polynomial.
    \item[(DR)] The equation (\ref{eqn1.5}) is density regular if and only if $b=a_1 +\cdots +a_s = 0$.
\end{enumerate}
\end{thm}

Note that, if $a_1 + \cdots + a_s \neq 0$, then Theorem \ref{thm1.3} implies that (\ref{eqn1.5}) is partition regular only if $P$ is an intersective polynomial. In the case where $a_1 + \cdots + a_s = 0$, we see that the set of solutions to (\ref{eqn1.5}) is unchanged if we replace $P$ by the intersective polynomial $P - P(0)$. Thus, intersectivity is absolutely vital for partition regularity, and is not merely a technical assumption in Theorem \ref{thm1.1}. Our results are definitive in this regard, and also definitive in terms of the coefficients $a_1,\ldots,a_s,b$.

In terms of the number of variables required, our results are state of the art in the sense that they match current progress on the asymptotic formula in Waring's problem. In the monomial case, the second author found with Lindqvist and Prendiville \cite{CLP2021} that $(1+o(1))d\log d$ variables suffice to characterise partition regularity. However, reducing the number of variables in that way requires estimates for moments of smooth Weyl sums that depend crucially on the multiplicative structure of the polynomial $P(x) = x^d$.

\subsection{Supersaturation}

Frankl, Graham, and R\"{o}dl \cite[Theorem 1]{FGR1988} obtained a stronger, quantitative version of Rado's theorem for systems of linear homogeneous equations. More precisely, they showed that, for a given partition regular system of linear equations and for sufficiently large $N$, a positive proportion of all solutions to the system over $\{1,\ldots,N\}$ become monochromatic under any $r$-colouring of $\{1,\ldots,N\}$. They also obtained an analogous result for density regular linear systems \cite[Theorem 2]{FGR1988}. This phenomenon, in which a positive proportion of solutions are found to be monochromatic or lie over an arbitrary dense set, is termed \emph{supersaturation}, in analogy with similar results from extremal combinatorics.

One significant corollary of supersaturation results is that one can obtain monochromatic solutions which are \emph{non-trivial}, in the sense that the variables of the solution are distinct. This may be readily deduced from supersaturation if one can first show that the set of trivial solutions is sparse in the set of all solutions. 

In previous work of the second author with Lindqvist and Prendiville \cite[Theorem 1.4]{CLP2021}, it was shown that partition regular equations of the form (\ref{eqn1.2}) with $P(x)=x^2 -1$ and $s\geqslant 5$ satisfy supersaturation. They also obtained similar results for partition regular linear homogeneous equations in logarithmically smoothed numbers \cite[Theorem 1.5]{CLP2021}. 
More recently, Prendiville \cite[Theorem 1.7]{Pre2021} has established supersaturation for partition regular equations (\ref{eqn1.2}) in the case where $P(x)=x^2$ and $s\geqslant 5$.

Our next theorem demonstrates that partition and density regular equations (\ref{eqn1.2}) in sufficiently many variables satisfy supersaturation. Furthermore, as per the remark above, we can ensure that the solutions we obtain are non-trivial.

\begin{thm}\label{thm1.4}
Let $d\geqslant 2$ be an integer, and define $s_1(d)$ by (\ref{eqn1.3}). 
Let $P$ be an intersective integer polynomial of degree $d$. Let $s\geqslant s_1(d)$ be an integer, and let $a_1,\ldots,a_s$ be non-zero integers. Given a set of integers $\cA$, write
\begin{equation*}
    \cS(\cA) := \{(x_{1},\ldots,x_s)\in\cA^s : x_i\neq x_j \text{ for all } i\neq j, \text{ and }
    a_1P(x_1) + \cdots + a_s P(x_s) =0\}.
\end{equation*}
\begin{enumerate}
    \item[(PR)]  If there exists a non-empty set $I\subseteq\{1,\ldots,s\}$ such that $\sum_{i\in I}a_i =0$, then for any positive integer $r$ there exists a positive real number $c_1(r) = c_1(P;a_1,\ldots,a_s;r)$ and a positive integer $N_1 = N_1(P;a_1,\ldots,a_s;r)$ such that the following is true for any positive integer $N\geqslant N_1$. Given any $r$-colouring $\{1,\ldots,N\}=\cC_1 \cup\cdots\cup \cC_r$, there exists $k\in\{1,\ldots,r\}$ such that
    $|\cS(\cC_k)|\geqslant c_1(r)N^{s-d}$.
    \item[(DR)] If $a_1 +\cdots +a_s = 0$, then for any positive real number $\delta>0$ there exists a positive real number $c_2(\delta) = c_2(P;a_1,\ldots,a_s;\delta)$ and a positive integer $N_2 = N_2(P;a_1,\ldots,a_s;\delta)$ such that the following is true for any positive integer $N\geqslant N_2$. Given any set $A\subseteq\{1,\ldots,N\}$ satisfying $|A|\geqslant\delta N$, we have $|\cS(A)|\geqslant c_2(\delta)N^{s-d}$.
\end{enumerate}
\end{thm}

\subsection{Linearised equations}

In the course of proving our main theorems, we are led to study certain `linearised' equations. These take the form
\begin{equation}\label{eqn1.6}
    L_1(\bn) = L_2(P(\bz)),
\end{equation}
where $P$ is an integer polynomial and $P(\bz)=(P(z_1),\ldots,P(z_{t}))$, for some non-degenerate linear forms $L_1$ and $L_2$ such that $L_1(1,\ldots,1)=0$. Here, a \emph{non-degenerate linear form in $t$ variables} $L:\Z^{t}\to\Z$ is a multilinear map of the form $L(\bx)=b_1 x_1 +\cdots +b_t x_t$, where $b_1,\ldots,b_t$ are non-zero integers. This naturally provides us with an opportunity to consider partition regularity criteria for such linearised equations (\ref{eqn1.6}). 

Recently, Prendiville \cite{Pre2021} studied the equation (\ref{eqn1.6}) in the case where $P(z)=z^2$, obtaining necessary and sufficient conditions for partition regularity as well as a counting result for certain partition regular equations of this form.
By incorporating Prendiville's `cleaving' strategy into our methods, we obtain a counting result on partition regularity for linearised equations (\ref{eqn1.6}) in sufficiently many variables.

\begin{thm}\label{thm1.5}
Let $d\geqslant 2$ and $r$ be positive integers, and let $s_1(d)$ be defined by (\ref{eqn1.3}).
Let $P$ be an intersective integer polynomial of degree $d$.
Let $s$ and $t$ be positive integers satisfying $s+t\geqslant s_1(d)$. Let $L_1$ and $L_2$ be non-degenerate linear forms in $s$ and $t$ variables respectively, and assume that $L_{1}(1,1,\ldots,1)=0$. There exists a positive constant $c_0=c_0(L_1,L_2,P,r)$ and a positive integer $N_{0}=N_{0}(L_1,L_2,P,r)$ such that the following holds for every positive integer $N\geqslant N_0$. For any $r$-colouring $\{1,\ldots,N\}=\cC_{1}\cup\cdots\cup\cC_r$, there exist $k\in\{1,\ldots,r\}$ such that, on writing $M:=N^{d^{-r}}$, we have
\begin{equation*}
     \{ (\bn,\bz) \in \cC_k^{s} \times \cC_k^{t}: L_1(\bn) = L_2(P(\bz)) \} \geqslant c_0 M^{d(s-1)+t}.
\end{equation*}
\end{thm}

As observed by Prendiville \cite{Pre2021}, bounds of this shape are sharp for generic linearisaed equations (\ref{eqn1.6}). Consider, for example, the equation
\begin{equation*}
    n_1 + \cdots + n_{s-1} - (s-1)n_s = s(z_1^d + \cdots + z_t^d).
\end{equation*}
If $(\bn,\bz)\in\{1,\ldots,N\}^{s+t}$ is a solution to this equation, then we see that $z_1,\ldots,z_t\leqslant N^{1/d}$. Consequently, if we colour $\{1,\ldots,N\}=\cC_{1}\cup\cdots\cup\cC_r$ by taking $\cC_r := \{1,\ldots,M^d\}$ and
\begin{equation*}
    \cC_i := \left\lbrace x\in\bN: N^{d^{-i}} <x\leqslant N^{d^{-(i-1)}}\right\rbrace \quad (1\leqslant i<r),
\end{equation*}
then we find that all monochromatic solutions to our equation come from $\cC_r$. We therefore conclude that there are at most $M^{d(s-1) +t}$ monochromatic solutions, which is within a constant factor of the lower bound given in Theorem \ref{thm1.5}.

\subsection{Methods}

As in the previous works \cite{Cha2022,Cho2018,CLP2021}, a key step in our argument is the application of a Fourier analytic \emph{transference principle}. The transference principle was originally developed by Green \cite{Gre2005A} to obtain solutions to linear equations in primes, and has subsequently been adapted to finding solutions over numerous different sets of arithmetic interest, such as the $k$th powers \cite[Part 2]{CLP2021}, logarithmically smooth numbers \cite[Part 3]{CLP2021}, and $k$th powers of primes \cite{Cho2018}.

If we assume that there exists a non-empty set $I\subseteq\{1,\ldots,s\}$ such that $\sum_{i\in I}a_i =0$, then we may rewrite (\ref{eqn1.2}) as
\begin{equation*}
    \sum_{i\in I}a_i P(x_{i}) = \sum_{j\in \{1,\ldots,s\}\setminus I}b_j P(x_{j}),
\end{equation*}
where $b_j = -a_j$ for all $j$. We then \emph{linearise} this equation to obtain a new equation
\begin{equation*}
    \sum_{i\in I}a_i n_{i} = \sum_{j\in \{1,\ldots,s\}\setminus I}b_j P_{D}(z_{j})
\end{equation*}
in variables $n_i$ and $z_j$, where $P_D$ is some auxiliary intersective polynomial (see (\ref{eqn3.9})). Counting solutions to this linearised equation may be accomplished more easily by using the arithmetic regularity lemma. We then use a transference principle to `transfer' solutions of the linearised equation to the original equation (\ref{eqn1.2}).

The main contribution of this article is the development of a transference principle for intersective polynomials. Given an intersective polynomial $P$, we consider a $W$-tricked version of the image set $\{P(1),P(2),\ldots\}$, namely, a set of the form
\begin{equation*}
    \cS_W=\left\{ \frac{P(x)-P(b)}{W_2}: x \equiv b \mmod{W_1} \right \},
\end{equation*}
for some $1\leqslant b\leqslant W_1$.
Here, $W$ is a product of powers of small primes, and $W\mid W_1 \mid W_2$ (see \S\ref{sec4} for further details). The upshot of working with this $W$-tricked set is that the elements of this new set are equidistributed in residue classes for small primes, whereas the original image set is not (consider, for example, the case where $P(X)=X^2$).

To establish the desired transference principle, we construct a pseudorandom majorant of the set $\cS_W$ defined above. This is carried out in \S\S\ref{sec4}--\ref{sec7}, and makes use of the Hardy--Littlewood circle method. In particular, we study the properties of exponential sums of the form 
\begin{equation*}
    \sum_{y\leqslant q}e_{q}(aP_{D}(y)),
\end{equation*}
where $P_D$ is some auxiliary intersective polynomial defined in terms of some parameter $D\in\N$ (in our applications, one can take $D=W^2)$, and as usual $e_q(ax):=\exp(2\pi i ax/q)$. One difficulty that arises here is that we require restriction estimates that are independent of the parameter $D$, though the coefficients of $P_D$ increase with $D$. To establish these uniform bounds, we exploit a key insight of Lucier~\cite{Luc2006}, that one can nevertheless bound the greatest common divisor of the coefficients of $P_D(X)-P_D(0)$ in terms of $P$ alone.

Finally, having applied the transference principle, it remains to prove that the linearised equation (\ref{eqn1.6}) admits many solutions $(\bn,\bz)$ with the $z_j$ lying in a (translated and dilated) colour class and the $n_i$ lying in a dense subset of $\{1,\ldots,N\}$. This is achieved by appealing to the arithmetic regularity lemma, as in \cite{Cha2022,Pre2021}.

We finish this subsection with a brief description of how we deal with the colouring aspect. There is an increasingly popular mantra that every colouring phenomenon is driven by an underlying density phenomenon. In the case of homogeneous equations, the connection was solved in practice by the second author with Lindqvist and Prendiville \cite{CLP2021} using \emph{homogeneous sets}. Subsequently, a full theoretical explanation was provided by the first author \cite{Cha2020}, who showed that homogeneous systems are partition regular if and only if they admit solutions with variables drawn from an arbitrarily given homogeneous set. A fresh hurdle that arises in the present work, compared with \cite{CLP2021}, is that our equation is inhomogeneous, and so the theory of homogeneous sets does not help us.
We resolve the issue by choosing the colour class which has the largest intersection with a certain polynomial Bohr set. For supersaturation, we supplement this with Prendiville's new cleaving technique, as alluded to earlier.

\subsection{Organisation}

We begin in \S\ref{sec2} by swiftly establishing necessary conditions for equations (\ref{eqn1.2}) and (\ref{eqn1.5}) to be partition or density regular. In particular, we prove the `only if' parts of Theorem \ref{thm1.1} and Theorem \ref{thm1.3}. We also give a short proof that Theorem \ref{thm1.4} implies Theorem \ref{thm1.1}, and Theorem \ref{thm1.1} implies Theorem \ref{thm1.3}.

In \S\ref{sec3}, we state the two main results which are the focus of this paper: Theorem \ref{thm3.4} and Theorem \ref{thm3.8}. We show that these two theorems imply all of our results stated above. We also recall some useful properties on intersective polynomials from \cite{Luc2006}, in particular, the notion of \emph{auxiliary intersective polynomials}.

The next four sections, \S\S\ref{sec4}--\ref{sec7}, are used to prove that Theorem \ref{thm3.4} follows from Theorem \ref{thm3.8}. In \S\ref{sec4}, we apply the $W$-trick and introduce the majorant $\nu$, the latter of which is the focus of our investigations in the next two sections. In \S\ref{sec5}, we use the Hardy--Littlewood circle method to establish a Fourier decay estimate for $\nu$. We continue in \S\ref{sec6} by establishing restriction estimates for $\nu$ and for a related majorant $\mu_D$. The conclusions of these three sections are combined in \S\ref{sec7} to apply a transference principle, which is used to complete the proof that Theorem \ref{thm3.8} implies Theorem \ref{thm3.4}.

Finally, in \S\ref{sec8}, we prove Theorem \ref{thm3.8} by using a version of Green's arithmetic regularity lemma.

We also include a section in the appendix on polynomial congruences, the results of which are used in \S\ref{sec4} to execute the $W$-trick.

\subsection*{Notation} 
Let $\N$ denote the set of positive integers. For each prime $p$, let $\bQ_p$ and $\bZ_p$ denote the $p$-adic numbers and the $p$-adic integers respectively. Given a real number $X>0$, we write $[X] := \{n\in\N:n\leqslant X\}$.
Set $\bT = [0,1]$.
For $q \in \bN$ and $x \in \bR$, we write $e(x) = e^{2 \pi i x}$ and $e_q(x) = e(x/q)$. 
For $P(x) \in \bZ[x]$ and $\bx = (x_1,\ldots,x_s)$, where $s \in \bN$, we abbreviate $P(\bx) = (P(x_1),\ldots,P(x_s))$. 
If $P$ is a polynomial with integer coefficients, we write $\gcd(P)$ for the greatest common divisor of its coefficients. The letter $\eps$ denotes a small, positive constant, whose value is allowed to differ between separate occurrences. We employ the Vinogradov and Bachmann--Landau asymptotic notations, with the implied constants being allowed to depend on $\eps$. 
For a finitely supported function $f: \bZ \to \bC$, the \emph{Fourier transform} $\hat{f}$ is defined by
\begin{equation*}
 \hat f(\alp) := \sum_{n \in \bZ} f(n) e(\alp n) \qquad (\alp \in \bR).   
\end{equation*}

\subsection*{Acknowledgements} 
This project began when JC visited SC at the University of Warwick as part of a London Mathematical Society Early Career Fellowship, and was completed when SC visited JC at the University of Bristol. We are grateful to these institutions for their hospitality and support.

We thank Christopher Frei, Sean Prendiville and Trevor Wooley for helpful conversations.

\subsection*{Funding}

JC was supported by an LMS Early Career Fellowship (07/2021 - 09/2021), and by the Heilbronn Institute for Mathematical Research (10/2021 - present). SC was supported by EPSRC Fellowship Grant EP/S00226X/2, and by the Swedish Research Council under grant no. 2016-06596. 

\section{Necessary conditions}\label{sec2}

In this section, we establish the necessary conditions for partition and density regularity. In particular, we prove the `only if' directions of Theorem \ref{thm1.1} and Theorem \ref{thm1.3}. We begin by noting that the necessary conditions for equations (\ref{eqn1.2}) and (\ref{eqn1.5}) to be partition or density regular are the same as those for linear homogeneous equations.

\begin{prop}\label{prop2.1}
Let $s\in\bN$. Let $a_{1},\ldots,a_{s}$ be non-zero integers, and let $P$ be an integer polynomial of positive degree.
\begin{enumerate}[\upshape(I)]
    \item\label{itemPR} If the equation (\ref{eqn1.2}) is partition regular, 
    then there exists a non-empty set $I\subseteq[s]$ such that $\sum_{i\in I}a_{i}=0$.
    \item\label{itemDR} If the equation (\ref{eqn1.2}) is density regular, then $a_{1} + \cdots + a_{s}=0$.
\end{enumerate}
\end{prop}

\begin{proof}
First suppose that (\ref{eqn1.2}) is partition regular. By replacing each $a_{i}$ with $-a_{i}$, we may assume that the leading coefficient of $P$ is positive. Thus, we can find $M\in\N$ such that the restriction of $P$ to the set $\{M,M+1,\ldots\}$ defines a strictly increasing function with image $S=\{P(M),P(M+1),\ldots\}\subseteq\N$. 

By Rado's criterion \cite[Satz IV]{Rad1933}, the conclusion of (\ref{itemPR}) holds if and only if the underlying linear equation $a_{1}x_{1} + \ldots + a_{s}x_{s}=0$ is partition regular. We establish the latter by using a trick of Lefmann \cite[Theorem 2.1]{Lef1991}. Suppose that we have a finite colouring $S=\cC_{1}\cup\cdots\cup \cC_{r}$. By our choice of $M$, this induces a finite colouring $\{M,M+1,\ldots\}=\cC_{1}'\cup\cdots\cup \cC_{r}'$ with $\cC_{i}':=\{x\geqslant M: P(x)\in \cC_{i}\}$. By considering a colouring where each element of $[M-1]$ receives a unique colour, partition regularity guarantees that (\ref{eqn1.2}) admits monochromatic solutions with respect to any finite colouring of the set $\{M,M+1,\ldots\}$. We can therefore find $i\in[r]$ such that (\ref{eqn1.2}) has a solution over $\cC_{i}'$, whence $a_{1}x_{1} + \cdots + a_{s}x_{s}=0$ has a solution over $\cC_{i}$. We have therefore proven that $a_{1}x_{1} + \cdots + a_{s}x_{s}=0$ is partition regular, as required.

Now suppose that (\ref{eqn1.2}) is density regular. Let $t\in\{0,1,\ldots,\deg(P)\}$ be such that $P(t)\neq 0$, and let $p$ be a prime satisfying $p>\deg(P) + |a_{1}| + \cdots + |a_{s}|$ and $p\nmid P(t)$. Since (\ref{eqn1.2}) is density regular, it has a solution over the set $\{n\in\N: n\equiv t \mmod{p}\}$. Thus, by reducing (\ref{eqn1.2}) modulo $p$, we observe that $p\mid (a_{1}+\cdots+a_{s})$. Our hypothesis on the size of $p$ therefore delivers the conclusion $a_{1} + \cdots + a_{s}=0$.
\end{proof}

\begin{prop}\label{prop2.2}
Let $s\in\N$, and let $P\in\Z[X]$ have positive degree. Let $a_{1},\ldots,a_{s}\in\Z\setminus\{0\}$ and $b\in\Z$. If the equation (\ref{eqn1.5}) is partition regular, then $b=(a_{1}+\cdots+a_{s})m$ for some $m\in\Z$ such that $P(X)-m$ is an intersective polynomial.
Furthermore, if (\ref{eqn1.5}) is density regular, then $b=a_{1}+\cdots+a_{s}=0$.
\end{prop}

\begin{proof}
Suppose that the equation (\ref{eqn1.5}) is partition regular.
Note that for any $q\in\N$, by partitioning $\N$ into distinct residue classes modulo $q$, the partition regularity of (\ref{eqn1.5}) implies that $b\equiv (a_{1}+\cdots+a_{s})m_{q}\pmod{q}$ for some $m_{q}\in\Z$. In particular, we see that every integer divisor of $(a_{1}+\cdots+a_{s})$ must also divide $b$, whence $b=(a_{1}+\cdots+a_{s})m$ for some $m\in\Z$.

Now observe that if $a_{1}+\cdots+a_{s}=0$, then $b=0$ and we could take $m$ to be any integer. In particular, we could chose $m=P(0)$ so that $P(X)-m$ is trivially intersective. 
Suppose then that $a_{1}+\cdots+a_{s}\neq 0$, whence the integer $m=b(a_{1}+\cdots+a_{s})^{-1}$ is uniquely defined.

Assume for a contradiction that $P(X)-m$ is not intersective. By the Chinese remainder theorem, we can find a prime $p$ and a positive integer $k$ such that $P(x)\equiv m \mmod{p^{k}}$ has no integer solutions $x$. Now choose $h\in\N$ such that $p^{h}\nmid(a_{1}+\cdots+a_{s})$. It follows that there does not exist $x\in\Z$ satisfying the congruence
\begin{equation*}
    a_{1}P(x) + \cdots + a_{s}P(x) \equiv b\mmod{p^{h+k}}.
\end{equation*}
Hence, there are no monochromatic solutions to (\ref{eqn1.5}) with respect to the finite colouring given by partitioning $\N$ into distinct residue classes modulo $p^{h+k}$. This contradicts the assumption that (\ref{eqn1.5}) is partition regular, so $P(X)-m$ must be intersective.

Finally, suppose that (\ref{eqn1.5}) is density regular. Since density regularity implies partition regularity, we deduce that $b=(a_1 + \cdots +a_s)m$ for some integer $m$ such that $P(X) - m$ is an intersective polynomial. Subtracting $b$ from both sides therefore reveals that (\ref{eqn1.5}) can be rewritten as
\begin{equation}\label{eqn2.1}
    a_{1}(P(x_1)-m) + \cdots + a_s(P(x_s)-m) = 0.
\end{equation}
The conclusion that $b=a_1+\cdots +a_s = 0$ now follows from Proposition \ref{prop2.1}.
\end{proof}

\begin{remark}
Observe that none of the results in this section make any assumptions on the number of variables $s$. The condition $s\geqslant s_1(d)$ introduced in Theorem \ref{thm1.1} is only used to find solutions to our equations, that is, to obtain sufficient conditions for partition or density regularity.
\end{remark}

With these necessary conditions established, we close this section by noting that Theorem \ref{thm1.4} implies Theorem \ref{thm1.1}, and Theorem \ref{thm1.1} implies Theorem \ref{thm1.3}.

\begin{proof}[Proof of Theorem \ref{thm1.1} given Theorem \ref{thm1.4}]
The `only if' parts of Theorem \ref{thm1.1} may be inferred from Proposition \ref{prop2.2}, whilst the remaining `if' statements follow from Theorem \ref{thm1.4}.
\end{proof}

\begin{proof}[Proof of Theorem \ref{thm1.3} given Theorem \ref{thm1.1}]
The `only if' parts of Theorem \ref{thm1.3} follow from Proposition \ref{prop2.2}. Similarly, the (DR) statement of Theorem \ref{thm1.3} follows immediately from Theorem \ref{thm1.1} and Proposition \ref{prop2.2}. Finally, it remains to show that, under the hypotheses of Theorem \ref{thm1.3}, if $b=(a_1 + \cdots +a_s)m$ for some integer $m$ such that $P(X) - m$ is an intersective polynomial, then (\ref{eqn1.5}) is partition regular. Rewriting
(\ref{eqn1.5}) as (\ref{eqn2.1}), the desired result now follows from Theorem \ref{thm1.1}.
\end{proof}

\section{Linear form equations}\label{sec3}

As we have verified the necessary conditions for partition and density regularity, our focus is now on obtaining solutions to (\ref{eqn1.2}) under the assumption that Rado's condition holds, meaning that there exists $I\subseteq[s]$ with $I\neq\emptyset$ such that $\sum_{i\in I}a_i = 0$. This condition allows us to rewrite (\ref{eqn1.2}) as
\begin{equation*}
    \sum_{i\in I}a_i P(x_{i}) = \sum_{j\in[s]\setminus I}(-a_j) P(x_j).
\end{equation*}
The above equation takes the shape
\begin{equation}\label{eqn3.1}
    L_{1}(P(\bx)) = L_2(P(\by))
\end{equation}
for some linear forms $L_1$ and $L_2$. 
We refer to a linear form $L(\bx)=b_1 x_1 + \cdots + b_t x_t$ in $t$ variables as \emph{non-degenerate} if $b_j\neq 0$ for all $j\in[t]$, and we write $\gcd(L):=\gcd(b_1,\ldots,b_t)$.

\begin{remark}
Note that in the above paragraph we could have $I=[s]$. We follow the convention that if we have an equation involving two linear forms $L_2$ where one of the forms has $t=0$ variables, then we replace $L_2$ with $0$. In particular, in this situation, the equation (\ref{eqn3.1}) takes the form
\begin{equation*}
    L_1(P(\bx)) = 0.
\end{equation*}
\end{remark}

To proceed further with our study of equations of the form (\ref{eqn3.1}), we require some notation.
Let  $T = T(d) \in \bN$ be minimal such that, for every integer polynomial $P$ of degree $d$, the equation
\begin{equation} \label{eqn3.2}
P(x_1) + \cdots + P(x_T) = P(x_{T+1}) + \cdots + P(x_{2T})
\end{equation}
has $O_{P}(X^{2T - d + \eps})$ solutions $\bx \in [X]^{2T}$, and let 
\begin{equation} \label{eqn3.3}
s_0(d) = 2T(d) + 1.
\end{equation}
The proof of \cite[Corollary 14.7]{Woo2019} yields
\[
T(d) \le \frac{d(d-1)}2 + \lfloor \sqrt{2d+2} \rfloor,
\]
and it follows from Hua's lemma \cite[Equation (1)]{Hua1938} that $T(2) \le 2$ and $T(3) \le 4$.
Hence
\[
s_0(d) \le s_1(d),
\]
where $s_1(d)$ is as in (\ref{eqn1.3}).
Moreover, by considering solutions with $x_{i}=x_{i+T}$ for all $i\in[T]$, we have 
\[
T(d) \geqslant d,
\qquad
s_0(d) \ge 2d + 1.
\]

By orthogonality, our definition of $T=T(d)$ is equivalent to the statement that
\begin{equation}\label{eqn3.4}
    \int_{\T}\left\lvert \sum_{x\leqslant X}e(\alpha P(x))\right\rvert^{2T}\ll_{P} X^{2T-d+\eps}
\end{equation}
holds for any integer polynomial $P$ of degree $d$. We now use this observation to bound the number of trivial solutions to (\ref{eqn1.2}) and (\ref{eqn1.5}).

\begin{lemma} \label{lem3.2} Let $d,s,X\in\N$ with $s\geqslant 2$, and let $P$ be an integer polynomial with degree $d$. Let $a_1,\ldots,a_s,b,c$ be fixed integers, and let $j,k \in [s]$ with $j \ne k$. If $s\geqslant s_{0}(d)$, then
\[
\# \{ \bx \in [X]^s: a_1 P(x_1) + \cdots + a_s P(x_s) = b,
\quad x_j = c \} \ll_P X^{s-d-1+\eps}
\]
and
\[
\# \{ \bx \in [X]^s: a_1 P(x_1) + \cdots + a_s P(x_s) = b,
\quad x_j = x_k \} \ll_P X^{s-d-1+\eps+d/(s-1)}.
\]
\end{lemma}

\begin{proof} For $\alp \in \bT$, write $f(\alp) = \sum_{x \le X} e(\alp P(x))$. By orthogonality, H\"older's inequality, and (\ref{eqn3.4}), we have
\begin{align*}
&
\#\{ \bx \in [X]^s: a_1 P(x_1) + \cdots + a_s P(x_s) = b,
\quad x_j = c \} \\
&\qquad =\int_\bT \left( \prod_{i \in [s] \setminus \{j \}} f(a_i \alp) \right) e(\alp(a_j  P(c) - b)) \d \alp 
\le \prod_{i \in [s] \setminus \{j \}} \left( \int_\bT |f(a_i \alp)|^{s-1} \right)^{1/(s-1)}
\\ & \qquad \le \prod_{i \in [s] \setminus \{j \}} 
\left( X^{s-1-2T} \int_\bT |f(a_i \alp)|^{2T} \right)^{1/(s-1)}
\ll_P X^{s-1-2T}X^{2T - d + \eps}
= X^{s-d-1+\eps}.
\end{align*}
Similarly, irrespective of whether $a_j + a_k$ vanishes, we have
\begin{align*}
&\# \{ \bx \in [X]^s: a_1 P(x_1) + \cdots + a_s P(x_s) = b,
\quad x_j = x_k \} 
\\ &\qquad = \int_\bT \left(
\prod_{i \in [s] \setminus \{j,k\}} f(a_i \alp) \right) f((a_j+a_k)\alp) e(-b\alp) \d \alp \\
& \qquad \le
\left( \prod_{i \in [s] \setminus \{j,k \}}  \int_\bT |f(a_i \alp)|^{s-1} \right)^{1/(s-1)}
\left( \int_\bT |f((a_j+a_k)\alp)|^{s-1} \d \alp \right)^{1/(s-1)} \\
&\qquad \ll_P
(X^{s-1-2T}X^{2T-d+\eps})^{(s-2)/(s-1)}X
\le X^{s - d - 1 + \eps + d/(s-1)}.
\end{align*}
\end{proof}

\begin{remark}
In our applications of Lemma \ref{lem3.2}, the quantity $X$ is chosen to be sufficiently large relative to $P$. Consequently, we can take $X^\eps$ sufficiently large such that the dependence on $P$ of the implicit constants can be removed.
\end{remark}

Let $P$ be an intersective integer polynomial of degree $d \ge 2$. To simplify our forthcoming arguments, we first restrict our attention to polynomials $P$ which are strictly monotone increasing and positive on the real interval $[1,\infty)$. That is, we assume $P$ satisfies
\begin{equation}
\label{eqn3.5}
  1\leqslant P(x) < P(y) \qquad(x,y\in\bR, \; 1\leqslant x<y).
\end{equation}
To prove the main theorems stated in the introduction, we prove the following counting result for equations of the form (\ref{eqn3.1}) where the $y_j$ variables lie in a particular colour class and the $x_i$ variables are drawn from an arbitrary dense set.

\begin{thm} \label{thm3.4}
Let $r$ and $d\geqslant 2$ be positive integers, and let $0<\delta<1$ be a real number. 
Let $P$ be an intersective integer polynomial of degree $d$ which satisfies (\ref{eqn3.5}).
Let $s \ge 1$ and $t \ge 0$ be integers such that $s + t \ge s_0(d)$. Let
\[
L_1(\bx) \in \bZ[x_1,\ldots,x_s], \qquad L_2(\by) \in \bZ[y_1,\ldots,y_t] 
\]
be non-degenerate linear forms such that $L_1(1,\ldots,1) = 0$. Let $X \in \bN$ be sufficiently large, and suppose $[X] = \cC_1 \cup \cdots \cup \cC_r$.
Then there exists $k \in [r]$
with $|\cC_k|\gg_{\delta,r,L_{1},L_2,P} X$
such that the following is true. For all $A\subseteq[X]$ with $|A|\geqslant\delta X$, we have
\begin{equation}\label{eqn3.6}
\# \{ (\bx, \by) \in A^s \times \cC_k^t: L_1(P(\bx)) = L_2(P(\by)) \} \gg X^{s+t-d}.
\end{equation}
The implied constant may depend on $L_1, L_2, P, r, \del$.
\end{thm}

\subsection{Deducing Theorem \ref{thm1.4}}

Having introduced Theorem \ref{thm3.4}, we now show how it can be used to prove Theorem \ref{thm1.4}.
Given a polynomial $P$ satisfying (\ref{eqn3.5}), we see that the conclusion of Theorem \ref{thm1.4} would follow immediately from Theorem \ref{thm3.4} if we could ensure the colour class $\cC_k$ we obtain has density at least $\delta$, as this would enable us to set $A=\cC_k$. Unfortunately, this cannot always be guaranteed. Nevertheless, the conclusion of Theorem \ref{thm3.4} informs us that $|\cC_k|\geqslant \delta_2 |X|$ for some $\delta_{2}\gg_{L_1,L_2,P,r,\delta}1$. We may therefore apply Theorem \ref{thm3.4} with this new density $\delta_2$ and find another colour class $\cC_{k_2}$. Iterating this argument eventually yields a colour class of sufficient density that our initial strategy of setting $A$ equal to a colour class can now be used to obtain Theorem \ref{thm1.4}.

The argument outlined above is termed \emph{cleaving} by Prendiville \cite{Pre2021}, who used this method to obtain a supersaturation result for the diagonal quadratic equations considered in \cite{CLP2021} (see \cite[\S2.1]{Pre2021} for an overview of the cleaving strategy in the context of Schur's theorem). We now use this argument to show that, for $X$ sufficiently large, there is a colour class $\cC_k$ such that the conclusion (\ref{eqn3.6}) of Theorem \ref{thm3.4} holds with $A=\cC_k$.

\begin{thm}\label{thm3.5}
Let $d$ and $r$ be positive integers, and let $P$ be an intersective integer polynomial of degree $d$ which satisfies (\ref{eqn3.5}).
Let $s \ge 1$ and $t \ge 0$ be integers such that $s + t \ge s_0(d)$. Let
\[
L_1(\bx) \in \bZ[x_1,\ldots,x_s], \qquad L_2(\by) \in \bZ[y_1,\ldots,y_t] 
\]
be non-degenerate linear forms such that $L_1(1,\ldots,1) = 0$. Let $X \in \bN$ be sufficiently large, and suppose $[X] = \cC_1 \cup \cdots \cup \cC_r$.
Then there exists $k \in [r]$ such that
\begin{equation*}
\# \{ (\bx, \by) \in \cC_k^s \times \cC_k^t: L_1(P(\bx)) = L_2(P(\by)) \} \gg X^{s+t-d}.
\end{equation*}
The implied constant may depend on $L_1, L_2, P, r, \del$. 
\end{thm}

\begin{proof}[Proof of Theorem \ref{thm3.5} given Theorem \ref{thm3.4}]
For each $\delta>0$, let $c_{0}(\delta)$ be the implicit constant appearing in the bound $|\cC_k|\gg_{L_1,L_2,,P,r,\delta} X$ in Theorem \ref{thm3.4}. Since decreasing the value of this constant does not invalidate the conclusion of Theorem \ref{thm3.4}, we may henceforth assume that $0<c_{0}(\delta)\leqslant\delta$ for all $\delta>0$.

Now set $\delta_0 = 1/r$ and let $\delta_{i} = c_{0}(\delta_{i-1})$ for all $i\in[r]$. By the pigeonhole principle, we can find $k_{0}\in[r]$ such that $|\cC_{k_{0}}|\geqslant X/r$. For all $i\in[r]$, let $k_{i}\in[r]$ be the index obtained by applying Theorem \ref{thm3.4} with $\delta=\delta_i$. By the pigeonhole principle, we can find $0\leqslant i < j\leqslant r$ such that $k_i = k_j=:k$. We claim that $\cC_k$ satisfies the conclusion of Theorem \ref{thm3.5}. Indeed, since $|\cC_{k}|\geqslant c_{0}(\delta_i)X\geqslant \delta_j X$, our choice of $i$ and $j$ ensures that $\cC_{k}=\cC_{k_{j}}$ satisfies (\ref{eqn3.6}) with $A=\cC_{k_i}=\cC_k$. This completes the proof of Theorem \ref{thm3.5} provided that we assume that $X$ is sufficiently large in terms of $L_{1},L_{2},P,r,$ and $\delta_{r}$, which is permissible as $\delta_{0},\ldots,\delta_{r}$ are all bounded away from $0$ in terms of $L_{1},L_{2},P,r$.
\end{proof}

\begin{proof}[Proof of Theorem \ref{thm1.4} given Theorem \ref{thm3.4}]
In this proof we allow all implicit constants to depend on the parameters $P,a_1,\ldots,a_s,r,\delta$, and assume that $N$ is sufficiently large with respect to these parameters.
In view of Lemma \ref{lem3.2}, Theorem \ref{thm1.4} is equivalent to the same statement with the condition `$x_i\neq x_j$ for all $i\neq j$' removed from the definition of $\cS(\cA)$. We therefore proceed to prove this equivalent version of Theorem \ref{thm1.4}. Note that 
\[
s \ge s_1(d) \ge s_0(d).
\]

We first consider polynomials $P$ satisfying (\ref{eqn3.5}). 
Note that the density statement (DR) follows immediately from Theorem \ref{thm3.4}. 
For the colouring statement (PR), observe that the existence of a non-empty set $I\subseteq[s]$ such that $\sum_{i\in I}a_i =0$ implies that we may express the equation (\ref{eqn1.2}) as a linear form equation (\ref{eqn3.1}) with $L_1(1,\ldots,1)=0$. The desired result may therefore be deduced from Theorem \ref{thm3.5}.

\bigskip
Having proven Theorem \ref{thm1.4} for $P$ satisfying (\ref{eqn3.5}), it remains to treat the general case.
By replacing each $a_i$ with $-a_i$ if necessary, it suffices to prove Theorem \ref{thm1.4} under the assumption that the leading coefficient of $P$ is positive. Hence, there exists $b\in\N$ such that the polynomial $\tilde{P}(X):=P(X+b)$ obeys (\ref{eqn3.5}). 
\bigskip

Now, given a colouring $[N]=\cC_{1}\cup\cdots\cup\cC_r$, we define a new colouring $[N-b]=\tilde{\cC}_1\cup\cdots\cup\tilde{\cC}_r$ by setting $\tilde{\cC}_i:=\{x-b:x\in\cC_i\setminus[b]\}$. By our proof of the special case above, for $N$ sufficiently large, we deduce that there exists $k\in[r]$ such that
\begin{equation*}
\# \{ \bz \in \tilde{\cC}_k^s : a_1 \tilde{P}(z_1) +\cdots + a_s \tilde{P}(z_s)=0 \} \gg N^{s-d}.
\end{equation*}
The partition result (PR) now follows by adding $b$ to each entry of every solution $\bz\in \tilde{\cC}_k^s$ found above to obtain $\gg N^{s-d}$ solutions to (\ref{eqn1.2}) over $\cC_k$.

Similarly, for the density statement (DR), we replace the $\delta$-dense set $A\subseteq[N]$ with the set $\tilde{A}=\{a-b:a\in A\setminus[b]\}$. As in the previous paragraph, we can find $\gg N^{s-d}$ solutions to $a_1 \tilde{P}(z_1) +\cdots + a_s \tilde{P}(z_s)=0$ over $\tilde{A}$, each of which lifts to a solution to (\ref{eqn1.2}) over $A$ by adding $b$ to each entry.
\end{proof}

\begin{remark} As is clear from the deduction above, we in fact establish our main results under the weaker assumption that $s \ge s_0$. This enables an immediate refinement in the number of variables required if a stronger upper bound for $T(d)$ is found. We will also see that one can replace $T(d)$ by $T(P)$, this being the least positive integer $T$ such that \eqref{eqn3.2} has $O_P(X^{2T-d+\eps})$ solutions $\bx \in [X]^{2T}$.
\end{remark}

\subsection{Auxiliary intersective polynomials}
Akin to \cite{CLP2021}, we prove Theorem \ref{thm3.4} by using a `linearisation' procedure so that we may obtain solutions to (\ref{eqn3.1}) by transferring solutions from the linearised equation of the form
\begin{equation}\label{eqn3.7}
    L_{1}(\bn) = L_2(P_D(\bz)),
\end{equation}
for some primitive linear form $L$ and some auxiliary integer polynomial $P_D$. The purpose of this subsection is to formally define the auxiliary polynomials that we use, as well as to state the linearised version of Theorem \ref{thm3.4}.

Let $P$ be an intersective integer polynomial of degree $d \in \bN$.
Recall from the introduction that this means that for each $n\in\N$ there exists $x\in\bZ$ such that $P(x)\equiv 0\mmod{n}$.
Furthermore, observe that the property of being intersective is equivalent to having $p$-adic zeros for every prime $p$. Thus, for each prime $p$, we fix $z_p \in \bZ_p$ such that $P(z_p)=0$. Let $m_p\geqslant 1$ be the multiplicity of $z_p$ as a zero of $P$ over $\Z_p$. This allows us to define a completely multiplicative function $\lambda:\N\to\N$ such that $\lambda(p)=p^{m_{p}}$ for all primes $p$. Explicitly, writing $\ord_p(D)$ for the multiplicity of $p$ in the prime factorisation of $D$, we have
\[
\lam(D) := \prod_p p^{m_p \ord_p(D)} \qquad(D\in\N).
\]
For later use, we record the following fact from \cite[Equation (73)]{Luc2006}:
\begin{equation}\label{eqn3.8}
   D \mid \lam(D) \mid D^d. 
\end{equation}
The Chinese remainder theorem shows that for each positive integer $D$ there is a unique integer
$r_D\in (-D,0]$ such that
\[
r_D \equiv z_p \mmod {p^{\ord_p(D)} \bZ_p}
\]
holds for all primes $p$.

With this notation in place, we can introduce the auxiliary polynomial
\begin{equation}\label{eqn3.9}
    P_D(x) := \frac{P(r_D + Dx)}{\lam(D)} \in \bZ[x].
\end{equation}
Observe that our choice of $r_D$ and $\lambda(D)$ ensures that $P_D$ is indeed a polynomial with integer coefficients. 
These auxiliary polynomials and the surrounding notation were introduced by Lucier \cite{Luc2006} and have subsequently become a standard tool when working with intersective polynomials. The significance of this construction stems from Lucier's result \cite[Lemma 28]{Luc2006} that the greatest common divisor of the coefficients of $P_{D}(X)-P_D(0)$ is uniformly bounded over all $D\in\N$ in terms of $P$ only. This observation is critical in our application of the circle method to exponential sums with intersective polynomial phases (see Lemma \ref{lem6.3}).

Before moving on, we note that $P_D$ is also intersective.

\begin{lemma}\label{lem3.7}
Let $P$ be an intersective integer polynomial of positive degree, and let $D$ be a positive integer. Then the auxiliary polynomial $P_D$ defined by (\ref{eqn3.9}) is intersective.
\end{lemma}
\begin{proof}
It suffices to prove that $P_D$ has a zero over $\bZ_p$ for every prime $p$.
Fix a prime $p$ and write $D=p^{k}M$, where $p\nmid M$ and $k\geqslant 0$. Our definition of $r_D$ implies that $r_D = z_p + p^k t$ for some $t\in\bZ_p$, and so $r_D + Dx = z_p + p^k(t+Mx)$ for all $x\in\bZ_p$. Since $M$ is a multiplicative unit in $\bZ_p$, we can find $x\in\bZ_p$ such that $t + Mx =0$, whence $P_D(x)=0$, as required.
\end{proof}

We now state a `linearised' version of Theorem \ref{thm3.4}.

\begin{thm} \label{thm3.8}
Let $r$ and $d\geqslant 2$ be positive integers, and let $0<\delta<1$ be a real number. 
Let $P$ be an intersective integer polynomial of degree $d$ which satisfies (\ref{eqn3.5}).
Let $s \ge 1$ and $t \ge 0$ be integers such that $s + t \ge s_0(d)$. Let 
$L_1(\bx) \in \bZ[x_1,\ldots,x_s]$
be a non-degenerate linear form for which $L_1(1,1,\ldots,1)=0$, and let $L_2(\by) \in \bZ[y_1,\ldots,y_t]$ be a non-degenerate linear form.
Let $D, Z \in \bN$ satisfy $Z \ge Z_0(D, r, \del, L_1, L_2, P)$, and set $N:=P_D(Z)$.
If $[Z] = \cC_1 \cup \cdots \cup \cC_r$, then there exists $k \in [r]$ such that the following is true. For all $\cA\subseteq[N]$ such that $|\cA|\geqslant\delta N$, we have
\begin{equation}\label{eqn3.10}
\# \{ (\bn,\bz) \in \cA^s \times \cC_k^t: L_1(\bn) = L_2(P_D(\bz)) \} \gg N^{s-1} Z^t.
\end{equation}
The implied constant may depend on $L_1, L_2, P, r, \del$. 
\end{thm}

\begin{remark} \label{rmk3.9}
Observe that, in contrast with Theorem \ref{thm3.4}, we have not specified a lower bound for the density of the colour class $\cC_k$ provided by Theorem \ref{thm3.8}. This is because such a conclusion follows automatically by a simple counting argument. Indeed, for $t\geqslant 1$, the cardinality appearing on the left-hand side of (\ref{eqn3.10}) is bounded above by
\begin{equation*}
    \sum_{z\in\cC_k}|\{ (\bn,\bz) \in [N]^s \times [Z]^t: L_1(\bn) = L_2(P_D(\bz)),\; z_{t} = z \}|
    \leqslant |\cC_{k}|N^{s-1}Z^{t-1}.
\end{equation*}
Thus, we see that $|\cC_k|\geqslant cZ$, where $c$ is the implicit constant in (\ref{eqn3.10}).
\end{remark}

Before moving on, we show that it suffices to prove Theorem \ref{thm3.8} under the assumption that $\gcd(L_1)=1$. In \S\ref{sec8}, we demonstrate the utility of this condition by parameterising solutions to (\ref{eqn3.7}) with $\bz$ fixed.

\begin{prop}\label{prop3.10}
Assume that Theorem \ref{thm3.8} is true in the cases where $\gcd(L_1)=1$. Then,
up to modifying the quantity $Z_0(D,r,\delta,L_1,L_2,P)$ and the implicit constant in (\ref{eqn3.10}), Theorem \ref{thm3.8} holds in general.
\end{prop}
\begin{proof}
Let $M=\gcd(L_1)$, and assume that $M>1$. Using (\ref{eqn3.8}), we can find $\kappa\in\N$ such that $\lambda(M)=M\kappa$. By \cite[Lemma 22]{Luc2006}, there exists an integer $m$ in the range $-M < m \leqslant 0$ such that $\lambda(M)P_{DM}(X) = P_{D}(m + MX) \in \bZ[X]$. Let $Z\in\N$ be sufficiently large, let $N=P_D(Z)$, and set
\begin{equation*}
    \tilde{Z} := \frac{Z - m}{M},
    \quad \tilde{N}:= P_{DM}(\tilde{Z})=\frac{P_D(m + M\tilde{Z})}{\lambda(M)}=\frac{N}{M\kappa}.
\end{equation*}
Finally, given an $r$-colouring $[Z]=\cC_1 \cup\cdots\cup \cC_r$, set $\tilde{\cC}_i:=\{z\in[\tilde{Z}]: m+Mz\in\cC_i\}$ for each $i\in[r]$. 

Let $\delta>0$, and let $L$ be the non-degenerate linear form satisfying $L_1=\gcd(L_1)L$. Let $k\in[r]$ be the index given by applying Theorem \ref{thm3.8} with the $r$-colouring $[\tilde{Z}]=\tilde{\cC}_1 \cup\cdots\cup \tilde{\cC}_r$ and with parameters $(L,DM,\delta/2)$ in place of $(L_1,D,\delta)$. Now given $\cA\subseteq[N]$ such that $|\cA|\geqslant\delta N$, we claim that there exists a set $\tilde{\cA}\subseteq[\tilde{N}]$ of the form $\tilde{\cA}=\{x\in[\tilde{N}]:(\kappa x + h)\in\cA\}$, for some integer $h$, such that $|\tilde{\cA}|\geqslant (\delta/2)\tilde{N}$. Assuming that this is true, we observe that for any $(\tilde{\bn},\tilde{\bz})\in\tilde{\cA}^{s}\times\tilde{\cC}_k^t$ satisfying
\begin{equation*}
    L(\tilde{\bn}) = L_2(P_{DM}(\tilde{\bz})),
\end{equation*}
the tuple $(\bn,\bz)=(\kappa\tilde{\bn} + h,M\tilde{\bz} +m)\in\cA^{s}\times\cC_k^t$ satisfies
\begin{equation*}
    L_1(\bn) = \lambda(M)L(\tilde{\bn}) = L_2(P_D(m + M\tilde{\bz})) = L_2(P_D(\bz)).
\end{equation*}
Since this map $(\tilde{\bn},\tilde{\bz})\mapsto(\bn,\bz)$ is injective, the desired bound (\ref{eqn3.10}) follows from our choice of $k$.

It only remains to establish the existence of the set $\tilde{\cA}$. 
By partitioning $[N]$ into residue classes modulo $\kappa$, the pigeonhole principle furnishes an integer $b$ in the range $0 \le b < \kap$ such that the set 
\[
\cB:=\{x\in[(N+b)/\kappa]: (\kappa x - b)\in\cA \}
\]
satisfies $|\cB|\geqslant \delta N/\kappa$. Note that, provided $N$ is sufficiently large, we have 
\[
(N+b)/\kappa > N/(M\kappa)=\tilde{N}.
\]
Hence, by partitioning $[(N+b)/\kappa]$ into intervals of length between $\tilde{N}/2$ and $\tilde{N}$, we deduce from the pigeonhole principle that there exists a translate of $\cB$ with density at least $\delta/2$ on $[\tilde{N}]$. We can therefore find an integer $h$ such that the set $\tilde{\cA}=\{x\in[\tilde{N}]:(\kappa x + h)\in\cA\}$ satisfies $|\tilde{\cA}|\geqslant (\delta/2)\tilde{N}$, completing the proof of the claim.
\end{proof}

\subsection{Deducing Theorem \ref{thm1.5}}

We close this section by demonstrating that Theorem \ref{thm1.5} follows from Theorem \ref{thm3.8}. We first state the following slightly more technical version of Theorem \ref{thm1.5}.

\begin{thm}\label{thm3.11}
Let $d\geqslant 2$ and $r$ be positive integers, and let $s_0(d)$ be defined by (\ref{eqn1.3}).
Let $P$ be an intersective integer polynomial of degree $d$ which has a positive leading coefficient.
Let $s$ and $t$ be positive integers satisfying $s+t\geqslant s_0(d)$. Let $L_1$ and $L_2$ be non-degenerate linear forms in $s$ and $t$ variables respectively, and 
assume that $L_1(1,1,\ldots,1)=0$. There exists a positive constant $c_0=c_0(L_1,L_2,P,\delta,r)$ and a positive integer $N_{0}=N_{0}(L_1,L_2,P,\delta,r)$ such that the following is true. Let $Z_{0}\geqslant N_0$ be a positive integer and set $Z_{i}=P(Z_{i-1})$ for all $1\leqslant i\leqslant r$. Then given any $r$-colouring $\{1,\ldots,Z_{r}\}=\cC_{1}\cup\cdots\cup\cC_r$, there exist $k,m\in\{1,\ldots,r\}$ and an interval of positive integers $I$ of length $Z_{m}$ such that
\begin{equation*}
     \{ (\bn,\bz) \in (\cC_k \cap I)^{s} \times (\cC_k \cap [Z_{m-1}])^{t}: L_1(\bn) = L_2(P(\bz)) \} \geqslant c_0 Z_{m-1}^{d(s-1)+t}.
\end{equation*}
\end{thm}

\begin{proof}[Proof of Theorem \ref{thm1.5} given Theorem \ref{thm3.11}]
By replacing $P$ and $L_2$ with $-P$ and $-L_2$ respectively if necessary, we may assume without loss of generality that the leading coefficient of $P$ is positive. Note that, if $Q$ is an integer polynomial of positive degree, then $Q(x+1)/Q(x)\to 1$ as $x\to\infty$. We therefore deduce that, provided $N$ is sufficiently large, there exists $Z_0\in\N$ such that $N/2 < Z_r \leqslant N$, where $Z_1,\ldots,Z_r$ are as defined in the statement of Theorem \ref{thm3.11}.
Moreover, if $N$ (and hence $Z_0$) is sufficiently large relative to $P$ and $r$, then we may assume that $Z_{r-m} \asymp_P N^{d^{-m}}$ for all $0\leqslant m\leqslant r$. Finally, since $M:=N^{d^{-r}}\ll_P Z_{j-1}$ for all $j\in[r]$, applying Theorem \ref{thm3.11} to the colouring $[Z_r]=(\cC_1 \cap [Z_r]) \cup\cdots\cup (\cC_r \cap [Z_r])$ establishes Theorem \ref{thm1.5}.
\end{proof}

As in our deduction of Theorem \ref{thm1.4} from Theorem \ref{thm3.4}, we prove Theorem \ref{thm3.11} from Theorem \ref{thm3.8} using Prendiville's cleaving method. The particular `multi-scale' cleaving argument we use is a variant of the proof of \cite[Theorem 8.1]{Pre2021}.

\begin{proof}[Proof of Theorem \ref{thm3.11} given Theorem \ref{thm3.8}]
Let $\eta(\delta)=\eta(L_1,L_2,P,r;\delta)>0$ be the implicit constant in (\ref{eqn3.10}). The conclusion of Theorem \ref{thm3.8} implies that we may assume that $\eta(\delta)$ is decreasing in $\delta$, and that $\eta(\delta)<\delta$ for all $0<\delta\leqslant 1$.

Let $\delta_r:=1/r$, and for each $i\in[r]$ set $\delta_{r-i} = \eta(\delta_{r-i+1})/2$. Note that $\delta_0 \leqslant \delta_1 \leqslant \ldots\leqslant \delta_r$. Let $Z_0,Z_1,\ldots,Z_r$ be as defined in the statement of Theorem \ref{thm3.11}, and assume that $Z_0$ is sufficiently large in terms of $L_1,L_2,P,r$. By our construction of the $\delta_i$, we may therefore assume that each $Z_i$ is sufficiently large relative to $\delta_i$.

For each $i\in\{0,1,\ldots,r\}$, let $k_i\in[r]$ be the index given by applying Theorem \ref{thm3.8} with parameters $(D,Z,\delta)=(1,Z_{i},\delta_i)$ to the colouring $[Z_i]=(\cC_1 \cap[Z_i])\cup\cdots\cup(\cC_r\cap[Z_i])$. By the pigeonhole principle, we can find $k\in[r]$ and $0\leqslant i<j\leqslant r$ such that $k=k_i = k_j$. 

Recall from Remark \ref{rmk3.9} that $|\cC_k \cap[Z_j]|\geqslant \eta(\delta_j)Z_j$. Hence, by partitioning $[Z_j]$ into intervals of lengths between $Z_{i+1}/2$ and $Z_{i+1}$, the pigeonhole principle furnishes an interval $I\subseteq[Z_j]$ of length $|I|=Z_{i+1}$ such that 
\begin{equation*}
    |\cC_k\cap I|\geqslant (\eta(\delta_j)/2)|I| = \delta_{j-1}Z_{i+1}\geqslant \del_i Z_{i+1}.
\end{equation*}
Let $h$ be the integer satisfying $I + h = [Z_{i+1}]$, whence $\cA:= h + (\cC_k\cap I) \subseteq[Z_{i+1}]$. By the translation invariance property $L_1(1,\ldots,1)=0$, observe that if $(\bn,\bz)\in \cA^s \times (\cC_k \cap [Z_i])^t$ is a solution to $L_1(\bn)=L_2(P(\bz))$, then $(\bn-h,\bz)\in (\cC_k\cap I)^s \times (\cC_k \cap [Z_i])^t$ is also a solution. 
Here, for $\bn=(n_1,\ldots,n_s)$, we have written $\bn - h =(n_1 - h,\ldots,n_s - h)$.
Setting $m=i+1$, our choice of $k=k_i$ therefore completes the proof.
\end{proof}

To summarise, we have now shown that all of our main results follow from Theorem \ref{thm3.4} and Theorem \ref{thm3.8}. The focus of the rest of this paper is on first showing how to deduce Theorem \ref{thm3.4} from Theorem \ref{thm3.8} and then, finally, proving Theorem \ref{thm3.8}. 

\section{Linearisation and the \texorpdfstring{$W$}{W}-trick}\label{sec4}
In this section we perform the preliminary manoeuvres needed to deduce Theorem \ref{thm3.4} from Theorem \ref{thm3.8}. Henceforth, until the end of \S\ref{sec7}, we fix the parameters
\begin{equation}\label{eqn4.1}
    \del \in (0,1],r,L_1,L_2,P
\end{equation}
appearing in the statement of Theorem \ref{thm3.4} and allow all implicit constants to depend on these parameters unless specified otherwise. In particular, we assume that $P$ is an intersective integer polynomial satisfying (\ref{eqn3.5}). 
Finally, let $X$ and $C$ be positive integers, sufficiently large in terms of the parameters (\ref{eqn4.1}). 

\subsection{The \texorpdfstring{$W$}{W}-trick}

There are two main obstacles which need to be overcome when attempting to replace the equation $L_1(P(\bx))=L_2(P(\by))$ appearing in Theorem \ref{thm3.4} with the linearised equation $L_1(\bn)=L_2(P_D(\bz))$ in Theorem \ref{thm3.8}. The first problem concerns the different scales of the variables in the latter equation. This issue is handled by considering a weighted count of solutions, which we address in the next subsection.
The second obstacle comes from the fact that, unlike $\N$, the image set $\{P(n):n\in\N\}$ is not equidistributed in residue classes modulo $p$ for arbitrary primes $p$. 
This problem can be ameliorated for small primes $p\leqslant w$, for some parameter $w$, using the \emph{$W$-trick}. This technique, originally developed by Green \cite{Gre2005A} to solve linear equations in primes, has subsequently become a standard tool for solving Diophantine equations over sparse arithmetic sets \cite{BP2017, Cha2022,Cho2018,CLP2021,Pre2021,Sal2020}.

Let $w$ be a positive integer which is large in terms of the quantity $C$, and assume that the positive integer $X$ is large in terms of $w$.
Define
\[
M = C d^2 10^{2w}, \qquad
W = \left(\prod_{p \le w}p \right)^{100dw},
\qquad
V = \sqrt{W}.
\]
Put
\[
D = W^2,
\]
and let $N,Z \ge 1$ be given by
\begin{equation}\label{eqn4.2}
N = P_D(Z),
\qquad Z = \frac{X - r_D}{D}.
\end{equation}
We assume that $Z$ is a positive integer; we will explain in \S \ref{sec7} why we are allowed to make this assumption.

Given $A\subseteq[X]$ with $|A|\geqslant\delta X$, for $R \in \bN$ and $b \in [R]$, denote
\[
A_{b,R} = \{ x \in A: x \equiv b \mmod R \}.
\]
Writing $(H,W)_d$ to denote the largest $m \in \bN$ for which $m^d \mid (H,W)$, Lemma \ref{lemA.5} implies that
\[
\delta X \leqslant |A| \le \sum_{\substack{b \in [W]: \\ (P'(b),W)_d \le M}} |A_{b,W}| + O(10^w W M^{-1/2}\lceil X/W\rceil).
\]
Here we have made use of the trivial bound $|A_{b,W}|\leqslant \lceil X/W\rceil$ for all $b$.
Note that, since $w$ is large relative to $C$, if $(P'(b),W)_d \le M$ then $(P'(b),W) \mid V$. As $10^w M^{-1/2} \le C^{-1/2}$ and $C$ is large in terms of $\delta$, we therefore have
\[
\delta X \ll  \sum_{\substack{b \in [W]: \\ (P'(b),W) \mid V}} |A_{b,W}|,
\]
and maximising yields $b_0$ for which
\[
|A_{b_0,W}| \gg \frac{\delta X}{W}, \qquad (P'(b_0),W) \mid V.
\]

Define $\kap \in \bN$ by
\[
W \kap (P'(b_0),W) = \lam(D).
\]
By pigeonholing, there exists $b\in[W\kappa]$ with $b \equiv b_0 \mmod W$ such that
\[
|A_{b,W\kap}|
\gg \frac{\delta X}{W \kap}.
\]
As
\[
(P'(b),W) = (P'(b_0),W) \mid V,
\]
we see that
\[
(P'(b),W\kap) = (P'(b),W) = (P'(b_0),W).
\]

Set
\begin{equation*}
  \cA =\left \{ \frac{P(x)-P(b)}{\lambda(D)}: x \in A_{b,W\kappa} \right \},  
\end{equation*}
noting from the Taylor expansion that $\cA \subset \bZ$. Now, for a given colouring $[X]=\cC_1\cup\cdots\cup\cC_r$, for each $i \in [r]$ let
\begin{equation*}
    \tilde \cC_i := \{ z \in [Z]: r_D + D z \in C_i \}.
\end{equation*}
Observe that if $(\bn,\bz)\in\cA^s\times \tilde{\cC}_k^t$ satisfies the linearised equation (\ref{eqn3.7}) then $(\bx,\by)\in A^s\times\cC_k^t$ satisfies the original equation (\ref{eqn3.1}), where
\begin{equation*}
    n_{i} = \frac{P(x_i)-P(b)}{\lambda(D)} \quad (1\leqslant i\leqslant s), \qquad
    y_{j} = r_D + Dz_j \quad (1\leqslant j\leqslant t).
\end{equation*}
Moreover, by passing from the set $\{P(x):x\in\N\}$ to the set
\begin{equation*}
    \left\{ \frac{P(x)-P(b)}{\lambda(D)}: x \equiv b \mmod{W\kappa} \right \},
\end{equation*}
we have achieved our goal of equidistribution modulo all primes up to $w$. Indeed, writing $x = W\kappa y +b$, our choice of $b$ and the Taylor expansion of $P$ demonstrate that
\begin{equation*}
    \frac{P(x)-P(b)}{\lambda(D)} \equiv \left(\frac{P'(b)}{(P'(b),W)}\right)y \mmod{p}
\end{equation*}
holds for any prime $p\leqslant w$. The bracketed factor is coprime to $p$, whence, as $y$ varies over residue classes modulo $p$, the polynomial on the left-hand side equidistributes over congruence classes modulo $p$.

\subsection{Constructing the weight function}
Having resolved the problem of equidistribution, we return to the problem of handling the different scales $N$ and $Z$ in Theorem \ref{thm3.8}. To proceed we construct the following weight function.
Given $A\subseteq[X]$ with $|A|\geqslant\delta X$, let $b\in[W\kappa]$ and $\cA\subseteq\bZ$ be as defined above. Define
\begin{equation}\label{eqn4.3}
\nu=\nu_b: \bZ \to [0,\infty),
\qquad
\nu(n) =
(P'(b), W)^{-1} \sum_{\substack{x \in (b,X] \\ x \equiv b \mmod{W \kap} \\ \frac{P(x) - P(b)}{\lambda(D)} = n}}
P'(x).
\end{equation}
Observe that $\nu$ is supported on $[N]$.

\begin{lemma}[Density transfer] 
\label{lem4.1} If $X$ and $N$ are sufficiently large in terms of $w$ and the fixed parameters in (\ref{eqn4.1}), then
\[
\sum_{n \in \cA} \nu(n) \gg N.
\]
In particular, the implicit constant does not depend on $w$.
\end{lemma}

\begin{proof}
Let $c=c(\delta)$ be a suitably small, positive constant. Then
\[
\sum_{x \in A_{b,W\kap}} P'(x)
\gg \sum_{y \le \frac{cX}{W\kap}} (W\kap y + b)^{d-1} \gg (W \kap)^{d-1}
\left( \frac{X}{W\kap} \right)^d = \frac{X^d}{W\kap}.
\]
Whence, for $X$ sufficiently large, we have
\[
(P'(b), W) \sum_{n \in \cA} \nu(n) = O((W \kap)^{d-1}) + \sum_{x \in A_{b,W\kap}} P'(x)
\gg \frac{X^d}{W\kap}.
\]
From the definition of $\kappa$, we therefore conclude that
\[
\sum_{n \in \cA} \nu(n) \gg \frac{X^d}{\lambda(D)} \gg N.
\]
\end{proof}

Similarly
\[
\lVert\nu \rVert_1 \asymp N.
\]

\section{Fourier decay}\label{sec5}
Having introduced the weight function $\nu$, we study the properties of its Fourier transform $\hat{\nu}$ using the Hardy--Littlewood circle method. Throughout this section, we fix $\nu=\nu_b$ as given by (\ref{eqn4.3}), for some $b\in[W\kappa]$. The main result of this section is the Fourier decay estimate
\begin{equation} \label{eqn5.1}
\| \hat \nu - \hat 1_{[N]} \|_\infty
\ll w^{\eps-1/d} N.
\end{equation}
\begin{remark}
Although we made a judicious choice of $b$ in the previous section to establish Lemma \ref{lem4.1}, the results of this and the next section remain true for arbitrary $b$ satisfying $(P'(b),W) \mid \sqrt W$. In particular, these results do not make reference to any sets $A$ or $\cA$.
\end{remark}

For all $\alpha \in \bT$, we have
\begin{align*}
(P'(b), W) \hat \nu(\alp) &= \sum_{\substack{x \in (b,X] \\ x \equiv b \mmod{W \kap}}} P'(x) e \left(\alp 
\frac{P(x) - P(b)}{\lambda(D)}
\right) \\
&= \sum_{W \kap y + b \in (b,X]} P'(W \kap y + b) e \left(\alp 
\frac{P(W \kap y + b) - P(b)}
{\lambda(D)}
\right).
\end{align*}
For $q \in \bN$, $a \in \bZ$ and $\bet \in \bR$, write
\[
S(q,a) = \sum_{x \le q} e \left( \frac{a(P(W \kap x + b) - P(b))}{q\lambda(D)} \right),
\qquad
I(\bet) = \int_0^{N} e( \bet \gam) \d \gam.
\]

\begin{lemma} [Major arc asymptotic] \label{lem5.2} Let $q \in \bN$, $a \in \bZ$, and suppose $\| q \alp\| = |q \alp - a|$. Then
\[
\hat \nu(\alp) = q^{-1} S(q,a) I(\alp - \tfrac{a}{q}) + O(X^{d-1}(q + N \| q \alp \| )).
\]
\end{lemma}

\begin{proof} Put $\beta = \alp - \frac{a}{q}$. Breaking the sum into residue classes modulo $q$ yields
\begin{align*}
(P'(b), W) \hat \nu(\alp) = O((W \kap)^{d-1}) + \sum_{x \le q} \:
\sum_{X_0 < z \le Y_0} 
& P'(W \kap qz + W \kap x + b) \\
& e \left( \left(\frac a q + \bet \right)
\frac{P(W \kap qz + W \kap x + b) - P(b)}{\lambda(D)}
\right),
\end{align*}
where
\[
X_0 = \frac{-(W \kap x + b)}{W \kap q}, \qquad
Y_0 = \frac{X - (W \kap x +b)}{W \kap q}.
\]
Taylor's theorem yields
\[
P(W \kap qz + W \kap x + b)
\equiv P(W \kap x + b) \mmod{q\lambda(D)},
\]
so
\begin{align*}
(P'(b), W) \hat \nu(\alp) &= O((W \kap)^{d-1}) + \sum_{x \le q}
e \left( \frac{a(P(W \kap x + b) - P(b))}{q\lambda(D)} \right)
\sum_{X_0 < z \le Y_0} \phi_x(z),
\end{align*}
where
\[
\phi_x(z) = P'(W \kap qz + W \kap x + b)
e \left( \bet
\frac{P(W \kap qz + W \kap x + b) - P(b)}{\lambda(D)}
\right).
\]
By Euler--Maclaurin summation \cite[Equation (4.8)]{Vau1997}, we have
\[
\sum_{X_0 < z \le Y_0} \phi_x(z) = \int_{X_0}^{Y_0} \phi_x(z) \d z + O(X^{d-1}(1 + N|\bet|)).
\]
The change of variables 
\[
\gam = \frac{P(W \kap qz + W \kap x + b) - P(b)} {\lambda(D)}
\]
now yields
\[
\sum_{X_0 < z \le Y_0} \phi_x(z) = (P'(b),W) q^{-1} I(\bet) + O(X^{d-1}(1 + N|\bet|)),
\]
completing the proof.
\end{proof}

We have the standard bound
\begin{equation} \label{eqn5.2}
I(\bet) \ll \min \{ N, \| \bet \|^{-1} \}.
\end{equation}
Note that
\[
S(q,a) = \sum_{u \le q} e_q (a \cP(x)),
\]
where
\[
\cP(x) = \frac{ P(W \kap x + b) - P(b)}{\lambda(D)} =: \sum_{j \le d} v_j x^j \in \bZ[x].
\]

\begin{lemma} \label{lem5.3} Suppose $(q,a) = 1$. Then
\[
S(q,a) \ll q^{1 + \eps - 1/d}.
\]
Further, if $(q,W) > 1$ then $S(q,a) = 0$. Finally, if $q \ge 2$ then
$q^{-1} S(q,a) \ll w^{\eps - 1/d}$.
\end{lemma}

\begin{proof} Write $q = q_1 q_2$, where $q_1$ is $w$-smooth and $(q_2, W) = 1$.
Then
\begin{align*}
S(q,a) &= \sum_{u_1 \le q_1} \sum_{u_2 \le q_2} e_{q_1 q_2} \left(a \cP(q_2 u_1 + q_1 u_2)\right)
= \sum_{u_1 \le q_1} \sum_{u_2 \le q_2} e_{q_1 q_2} \left(a \sum_{j \le d} v_j (q_2 u_1 + q_1 u_2)^j \right) \\
&= \sum_{u_1 \le q_1} \sum_{u_2 \le q_2} e_{q_1 q_2} \left(a \sum_{j \le d} v_j ((q_2 u_1)^j + (q_1 u_2)^j) \right) = S(q_1,a_1) S(q_2,a_2),
\end{align*}
where
\[
q_2 a_1 \equiv a \mmod{q_1},
\qquad
q_1 a_2 \equiv a \mmod{q_2}.
\]

Put
\[
h = (q_1,W), \qquad q_1 = hq', \qquad W = hW'.
\]
Then
\begin{align*}
S(q_1,a_1) &= \sum_{u_1 \le q'} \sum_{u_2 \le h} e_{hq'}(a_1 \cP(u_1 + q' u_2))
= \sum_{u_1 \le q'} \sum_{u_2 \le h} e_{hq'} \left(a_1 \sum_{j \le d} v_j (u_1 + q' u_2)^j \right).
\end{align*}
As
\[
W \mid v_j \qquad (2 \le j \le d),
\]
we have
\[
S(q_1,a_1) = \sum_{u_1 \le q'} e_{q_1}(a_1 \cP(u_1)) \sum_{u_2 \le h} e_h(a_1 v_1 u_2).
\]
Observe that $v_1 = \frac{P'(b)}{(P'(b),W)}$, and recall that $(P'(b),W) \mid V = \sqrt W$. For each prime $p \le w$, we have $\ord_p(P'(b)) < \ord_p(W)$, so $\ord_p(v_1) = 0$. Thus $(v_1,W) = 1$, and in particular $(h,a_1v_1) = 1$. Hence
\[
S(q_1,a_1) = 
\begin{cases}
1, &\text{if }q_1 = 1 \\
0, &\text{if } q_1 \ne 1.
\end{cases}
\]
This shows that if $(q,W) > 1$ then $S(q,a) = 0$.

Next, we estimate
\[
S(q_2,a_2) = \sum_{x \le q_2} e_{q_2} \left(a_2 \sum_{j \le d} v_j x^j \right).
\]
The binomial theorem tells us that
\[
v_d = \frac{\ell_P (W \kap)^d}{\lam(D)} 
= \frac{\ell_P (W \kap)^{d-1}}{(P'(b),W)},
\]
where $\ell_P$ is the leading coefficient of $P$. As $(q_2,W) = 1$, we have in particular $(v_d,q_2) \ll 1$. Thus, by periodicity and \cite[Theorem 7.1]{Vau1997}, we have
\[
|S(q,a)| \le |S(q_2,a_2)| \ll q_2^{1+\eps-1/d} \le q^{1+\eps-1/d}.
\]

If $q \ge 2$ and $S(q,a) \ne 0$ then $q_1 = 1$ and $q_2 \ge 2$, whereupon $q_2 > w$ and
\[
q^{-1} S(q,a) = q_2^{-1} S(q_2,a_2) \ll q_2^{\eps-1/d} < w^{\eps-1/d}.
\]
\end{proof}

We establish \eqref{eqn5.1} using the circle method. Put $\tau = 1/100$ and $Q = X^\tau$. For coprime $q,a \in \bZ$ such that $0 \le a \le q \le Q$, define
\[
\fM(q,a) =  \{ \alp \in \bT: |\alp - a/q| \le Q/N \}.
\]
Let $\fM$ be the union of the sets $\fM(q,a)$, and put $\fm = \bT \setminus \fM$.

First suppose $\alp \in \fm$. By Dirichlet's approximation theorem (see \cite[Lemma 2.1]{Vau1997}), there exist coprime $q \in \bN$ and $a \in \bZ$ such that $q \le Q$ and $|\alp - a/q| \le (q Q)^{-1}$. As $\alp \in \fm$, we must also have $|q\alp - a| > q Q/N$, so
\[
\hat 1_{[N]}(\alp) \ll \| \alp \|^{-1} \le \frac{q}{\| q \alp \|} = \frac{q}{|q \alp - a|} < \frac{N}{Q}.
\]
By partial summation, we have
\[
(P'(b),W) \hat \nu(\alp) \ll X^{d-1} \sup_{Y \le X/(W\kap)}
\left| \sum_{y \le Y} e(\alp \cP(y)) \right|.
\]
By Dirichlet's approximation theorem, there exist coprime $v \in \bN$ and $b \in \bZ$ such that $v \le N/Q$ and $|\alp - b/v| \le Q/(vN)$. As $\alp \in \fm$, we must have $v > Q$. As $v_d \ll_w 1$, Weyl's inequality in the form \cite[Proposition 4.14]{Ove2014} yields
\[
\sum_{y \le Y} e(\alp \cP(y)) \ll_w
Y^{1+\eps}(Y^{-1} + v^{-1} + vY^{-d})^{2^{1-d}} \ll X^{1+\eps-\tau 2^{1-d}}.
\]
As $X$ is large in terms of $w$, we thus have
\[
\hat \nu(\alp) \ll N^{1+\eps-2^{1-d}/(100d)} \ll w^{\eps-1/d}N,
\]
wherein the implied constants do not depend on $w$, and hence
\begin{equation} \label{eqn5.3}
\hat \nu(\alp) - \hat 1_{[N]}(\alp) \ll w^{\eps-1/d} N.
\end{equation}

Next, suppose $a \in \{0,1\}$ and $\alp \in \fM(1,a)$. Then, by Lemma \ref{lem5.2}, we have 
\[
\hat \nu(\alp) = I(\alp) + O(X^{d-1+\tau}).
\]
Euler--Maclaurin summation yields
\[
I(\alp) - \hat 1_{[N]}(\alp) \ll 1 + N \| \alp \| \ll Q,
\]
so by the triangle inequality
\[
\nu(\alp)  - \hat 1_{[N]}(\alp) \ll  X^{d-1 + \tau} \ll w^{\eps-1/d} N.
\]

Finally, suppose $\alp \in \fM(q,a)$ with $q \ge 2$. Then $\| \alp \| \ge q^{-1} - |\alp - a/q| \gg q^{-1}$, so
\[
\hat 1_{[N]}(\alp) \ll \| \alp \|^{-1} \ll Q.
\]
By Lemmas \ref{lem5.2} and \ref{lem5.3}, as well as \eqref{eqn5.2}, we have
\[
\hat \nu(\alp) \ll w^{\eps-1/d} N.
\]
Thus, we again have \eqref{eqn5.3}. We have secured \eqref{eqn5.3} in all cases, completing the proof of \eqref{eqn5.1}.

We record, for later use, the following bounds that arose above.

\begin{lemma} \label{lem5.4} For $\tau=1/100$, we have
\[
\hat \nu(\alp) \ll 
N^{1+\eps-2^{1-d}/(100d)}
\qquad (\alp \in \fm)
\]
and
\[
\hat \nu(\alp) \ll 
q^{\eps-1/d} \min \{N, \| \alp - a/q \|^{-1} \} + X^{d-1+\tau}
\qquad (\alp \in \fM(q,a) \subset \fM).
\]
\end{lemma}

\section{Restriction estimates}\label{sec6}

Continuing our study of $\nu=\nu_b$ for fixed $b$, in this section we establish restriction estimates for $\nu$. We also obtain restriction estimates for a related weight function $\mu_D$ corresponding to the auxiliary polynomial $P_D$.

\subsection{Restriction for \texorpdfstring{$\nu$}{nu}}

Recall that $T=T(d)\in\N$ is as defined in \S\ref{sec3}.

\begin{lemma} \label{lem6.1} Let $E > 2T$ be real, and let $\phi: \bZ \to \bC$ with $|\phi| \le \nu$. Then
\[
\int_\bT |\hat \phi(\alp)|^E \d \alp \ll_E N^{E-1}.
\]
\end{lemma}

\begin{proof}
Note that $\lVert \phi\rVert_{\infty}\leqslant\lVert\nu\rVert_{\infty}\ll X^{d-1}$. Hence, by orthogonality and the triangle inequality, we have
\begin{align*}
\int_\bT |\hat \phi(\alp)|^{2T} \d \alp &= \sum_{n_1 + \cdots + n_T = n_{T+1} + \cdots + n_{2T}} \phi(n_1) \cdots \phi(n_T) \overline{\phi(n_{T+1}) \cdots \phi(n_{2T})} \\
&\leqslant \lVert \phi\rVert_{\infty}^{2T} \sum_{\substack{x_1,\ldots,x_{2T} \in [X] \\ P(x_1) + \cdots + P(x_T) = P(x_{T+1}) + \cdots + P(x_{2T})}}1
\\ &\ll X^{2T(d-1)+2T-d+\eps}
\ll N^{2T-1+\eps}.
\end{align*}
Let $u = (2T + E)/2$, in order to be sure that $u > 2d$. As
\[
\| \hat \phi \|_\infty
\le \| \nu \|_1 \ll N,
\]
we have
\[
\int_\bT |\hat \phi(\alp)|^u \d \alp \ll N^{u-1+\eps}.
\]
To complete the proof, we insert this almost-sharp moment estimate, together with the ingredients in Lemma \ref{lem5.4}, into the general epsilon-removal lemma \cite[Lemma 25]{Sal2020}.
\end{proof}

\begin{lemma}
\label{lem6.2}
Let $E > 2T$ be real, and let $\phi: \bZ \to \bC$ with $|\phi| \le \nu + 1_{[N]}$. Then
\[
\int_\bT |\hat \phi(\alp)|^E \d \alp \ll_E N^{E-1}.
\]
\end{lemma}

\begin{proof} We decompose $\phi = \phi_1 + \phi_2$, where $|\phi_1| \le \nu$ and $|\phi_2| \le 1_{[N]}$. Then
\begin{align*}
\int_\bT |\hat \phi_2(\alp)|^E \d \alp 
&\le N^{E-2T} \int_\bT |\hat \phi_2(\alp)|^{2T} \d \alp 
\\ &\le N^{E-2T} \sum_{n_1 + \cdots + n_T = n_{T+1} + \cdots + n_{2T}} \phi_2(n_1) \cdots \phi_2(n_T) \overline{\phi_2(n_{T+1}) \cdots \phi_2(n_{2T})}
\\ & \le N^{E-1}.
\end{align*}
By Lemma \ref{lem6.1} and the triangle inequality, we thus have
\[
\int_\bT |\hat \phi(\alp)|^E \d \alp \ll_E
\int_\bT |\hat \phi_1(\alp)|^E \d \alp
+ \int_\bT |\hat \phi_2(\alp)|^E \d \alp
\ll_E N^{E-1}.
\]
\end{proof}

\subsection{Restriction for \texorpdfstring{$P_D$}{PD}}

In this subsection, we fix some $D\in\N$, and let $N,Z\in\N$ be as in (\ref{eqn4.2}), for some $X$ which is sufficiently large relative to $P$ and $D$. Recalling (\ref{eqn3.5}), define
\begin{equation}\label{eqn6.1}
  \mu_D: \bZ \to \bC,
\qquad \mu_D(n) = \frac{N}{Z} \sum_{\substack{z \le Z \\ P_D(z) = n}} 1
=\frac{N}{Z}1_{P_{D}([Z])}(n).
\end{equation}
Note that $\mu_D$ is supported on $[P_{D}(Z)]=[N]$, and
that
\[
\| \mu_D \|_1 = N,
\qquad
\hat \mu_D(\alp) = \frac{N}{Z}\sum_{z \le Z} e(\alp P_D(z)).
\]
The purpose of this subsection is to establish the following restriction estimate for $\mu_D$, which is analogous to the bound obtained in Lemma \ref{lem6.1} for $\nu$.

\begin{lemma} \label{lem6.3} Let $E > 2T$ be real, and let $\phi: \bZ \to \bC$ with $|\phi| \le \mu_D$. Then
\[
\int_\bT |\hat \phi(\alp)|^E \d \alp \ll_E N^{E-1}.
\]
\end{lemma}

This is more subtle than Lemma \ref{lem6.1}, as it requires us to extract savings depending on the size of the coefficients of $P_D$. Indeed, it would be false if $\gcd(P_D -P_D(0))$ were large. However, we know from \cite[Lemma 28]{Luc2006} that 
\begin{equation} \label{eqn6.2}
\gcd(P_D - P_D(0)) \ll_P 1.
\end{equation}

Our proof of Lemma \ref{lem6.3} proceeds along similar lines to that of Lemma \ref{lem6.1}. Orthogonality yields
\begin{align*}
\int_\bT |\hat \phi(\alp)|^{2T} \d \alp 
&= (N/Z)^{2T} \sum_{\substack{z_1,\ldots,z_{2T} \le Z \\
P_D(z_1) + \cdots + P_D(z_T) = P_D(z_{T+1}) + \cdots + P_D(z_{2T})}} 1 \\ &\le
(N/Z)^{2T} \sum_{\substack{x_1,\ldots,x_{2T} \le X \\
P(x_1) + \cdots + P(x_T) = P(x_{T+1}) + \cdots + P(x_{2T})}} 1
\ll (N/Z)^{2T} 
X^{2T - d + \eps} \\
&\ll (N/Z)^{2T} (DZ)^{2T}
(DZ)^{\eps - d} \ll 
N^{2T} D^{2T} N^{(\eps-d)/d}
= D^{2T} N^{2T - 1 + \eps/d}.
\end{align*}
As $N$ is arbitrarily large compared to $D$, we thus have
\[
\int_\bT |\hat \phi(\alp)|^{2T} \d \alp \ll N^{2T - 1 + \eps},
\]
and the implied constant does not depend on $D$.
Let $u = (2T + E)/2$, in order to be sure that $u > 2d$. Since
\[
\| \hat \phi \|_\infty
\le \| \mu_D \|_1 = N,
\]
we have
\begin{equation} \label{eqn6.3}
\int_\bT |\hat \phi(\alp)|^u \d \alp \ll N^{u-1+\eps}.
\end{equation}

\bigskip

Suppose $\alp \in \fm$. By Dirichlet's approximation theorem, there exist coprime $v \in \bN$ and $b \in \bZ$ such that $v \le N/Q$ and $|\alp - b/v| \le Q/(vN)$. As $\alp \in \fm$, we must also have $v > Q$. The leading coefficient of $P_D$ is 
\[
\frac{\ell_P D^d}{\lam(D)} \ll_D 1, 
\]
where $\ell_P$ is the leading coefficient of $P$. Hence, Weyl's inequality in the form \cite[Proposition 4.14]{Ove2014} gives
\[
\hat \mu_D(\alp) \ll_D N^{1+\eps}(v^{-1} + Z^{-1} + vZ^{-d})^{2^{1-d}} \ll N^{1+\eps} (D^d/Q)^{2^{1-d}}.
\]
Since $N$ is arbitrarily large compared to $D$, we have
\begin{equation} \label{eqn6.4}
\hat \mu_D(\alp) \ll N^{1+\eps} Q^{-2^{1-d}},
\end{equation}
and the implied constant does not depend on $D$.

\bigskip

We come to the major arcs. For $q \in \bN$, $a \in \bZ$ and $\bet \in \bR$, define
\[
S_D(q,a) = \sum_{x \le q} e_q(a P_D(x)),
\qquad
I_D(\bet) = \frac{N}{Z} \int_0^Z e(\bet P_D(z)) \d z.
\]

\begin{lemma}\label{lem6.4}
Let $q \in \bN$, $a \in \bZ$, and suppose $\|q \alp\| = |q \alp - a|$. Let $\bet = \alp - \frac{a}{q}$. Then
\[
\hat \mu_D(\alp) = q^{-1} S_D(q,a) I_D(\bet) + O((q + N \| q \alp \|) N/Z).
\]
\end{lemma}

\begin{proof} Breaking the sum into residue classes modulo $q$ yields
\[
\hat \mu_D(\alp) = \frac{N}{Z} \sum_{y \le q} \sum_{X_0 < x \le Y_0} e(\alp P_D(qx + y)),
\]
where
\[
X_0 = -y/q, \qquad Y_0 = (Z-y)/q.
\]
By periodicity, we have $e_q(aP_{D}(qx+y))=e_q(aP_{D}(y))$, whence
\[
\hat \mu_D (\alp) = \frac{N}{Z} \sum_{y \le q} e_q(a P_D(y)) \sum_{X_0 < x \le Y_0} e(\bet P_D(q x + y)).
\]
Using Euler--Maclaurin summation, we find that
\begin{align*}
&\sum_{X_0 < x \le Y_0} 
e(\bet P_D(q x + y))
- \int_{X_0}^{Y_0} 
e(\bet P_D(q x + y)) \d x 
\\ &\ll 1 + \frac{ Y_0 q |\bet| 
Z^{d-1} D^d} {\lam(D)}
\ll 1 + \frac{Z|\bet| 
X^{d-1}D} {\lam(D)}
\ll 1 + N |\bet|.
\end{align*}
A change of variables gives
\[
\int_{X_0}^{Y_0} e(\bet P_D(q x + y)) \d x
= \frac{Z}{Nq} I_D(\bet),
\]
completing the proof.
\end{proof}

\begin{lemma}\label{lem6.5}
Suppose $(q,a) = 1$. Then
\[
S_D(q,a) \ll q^{1+\eps-1/d}.
\]
\end{lemma}

\begin{proof} In view of \eqref{eqn6.2}, this follows from periodicity and \cite[Theorem 7.1]{Vau1997}.
\end{proof}

\begin{lemma}\label{lem6.6}
We have
\[
I_D(\bet) \ll N(1 + N \| \bet \|)^{-1/d}.
\]
\end{lemma}

\begin{proof} 
The leading coefficient of $P_D$ is $\ell_P D^d/\lam(D)$, so by \cite[Theorem 7.3]{Vau1997} we have
\[
I_D(\bet) \ll N(1 + Z^d D^d \| \bet \| / \lam(D))^{-1/d}.
\]
The claimed bound follows upon noting that
\[
\frac{(ZD)^d}{\lam(D)} \asymp \frac{P(X)}{\lam(D)} = P_D(Z) = N.
\]
\end{proof}

\begin{proof}[Proof of Lemma \ref{lem6.3}]
Combining the results of Lemmas \ref{lem6.4}, \ref{lem6.5} and \ref{lem6.6} furnishes
\begin{equation} \label{eqn6.5}
\hat \mu_D (\alp) \ll \frac{q^{\eps-1/d} N}{(1 + N |\alp - a/q|)^{1/d}} + DNX^{2\tau-1}
\qquad (\alp \in \fM(q,a) \subset \fM).
\end{equation}
Finally, observe from its proof that \cite[Lemma 25]{Sal2020} holds with $\| \alp - a/q \|^\kap$ in place of $\| \alp - a/q \|$ in its third assumption.
Inserting \eqref{eqn6.3}, \eqref{eqn6.4} and \eqref{eqn6.5} into this, applied with $\kap = d^{-1} - \eps$, completes the proof.
\end{proof}

\section{The transference principle}\label{sec7}

Let $s \ge 1$ and $t \ge 0$ be integers such that $s + t \ge s_0(d)$, where $s_0(d)$ is as defined in \eqref{eqn3.3}. For finitely supported $f_1 ,\ldots, f_s: \Z \to \R$ and $h_1 ,\ldots, h_t: \Z \to \R$, define
\begin{equation}\label{eqn7.1}
 \Phi(f_1 ,\ldots, f_s;h_1,\ldots,h_t) = \sum_{L_1(\bn) = L_2(\mathbf{m})} f_1(n_1)\cdots f_s(n_s)h_1(m_1)\cdots h_t(m_t).   
\end{equation}
For finitely supported $f,h: \bZ \to \bR$, we abbreviate 
\[
\Phi(f_1,\ldots, f_s;h) := \Phi(f_1,\ldots,f_s;h,\ldots,h),
\qquad
\Phi(f;h) := \Phi(f,\ldots,f;h,\ldots,h).
\]
Given a finite set of integers $A$, we also write $\Phi(A;h):=\Phi(1_A;h)$.

We begin by showing that the size of counting operator $\Phi(f_1,\ldots,f_s;h)$ is controlled by the size of the Fourier coefficients of each of the $f_j$.
Write
\[
L_1(\bx) = a_1 x_1 + \cdots + a_s x_s, \qquad
L_2(\bx) = c_1 x_1 + \cdots + c_t x_t.
\]

\begin{lemma} [Fourier control] \label{lem7.1} Let $f_{1},\ldots,f_s :\Z\to\R$ and $h:\Z\to\R$. If $|h| \leqslant \mu_D$ and $|f_j| \leqslant \nu + 1_{[N]}$ for all $j\in[s]$, then
\[
\Phi(f_1,\ldots,f_s;h) \ll N^{s+t-1}\prod_{j\leqslant s}(\| \hat f_j \|_\infty/N)^{1/(2s+2t)} .
\]
\end{lemma}

\begin{proof} By orthogonality,  H\"older's inequality, and periodicity, we have
\begin{align*}
|\Phi(f_1,\ldots,f_s;h)| &= 
\left|
\int_\bT \prod_{j \le s} \hat f_j(a_j \alp) \cdot \prod_{\ell \le t} \hat h(-c_\ell \alp) \d \alp \right| 
\\ &\le \prod_{j \le s} \left( \int_\bT |\hat f_j (a_j \alp)|^{s+t} \d \alp\right)^{1/(s+t)} \cdot
\prod_{\ell \le t} 
\left(
\int_\bT |\hat h (-c_\ell \alp)|^{s+t} \d \alp\right)^{1/(s+t)}
\\ &= 
\prod_{j \le s}
\left( 
\int_\bT |\hat f_j(\alp)|^{s+t} \d \alp\right)^{1/(s+t)}
\cdot
\left( \int_\bT |\hat h (\alp)|^{s+t} \d \alp\right)^{t/(s+t)}
\\ &\le 
\left( \int_\bT |\hat h (\alp)|^{s+t} \d \alp\right)^{t/(s+t)}
\prod_{j \leqslant s}
\left( \|\hat f_j\|_\infty^{1/2} 
\int_\bT |\hat f_j (\alp)|^{s+t-1/2} \d \alp\right)^{1/(s+t)}.
\end{align*}
Lemmas \ref{lem6.2} and \ref{lem6.3} now give
\begin{align*}
\Phi(f_1,\ldots,f_s;h) &\ll (N^{s+t-1})^{t/(s+t)}\prod_{j\leqslant s}\left(\|\hat f_j\|_\infty^{1/(2s+2t)}N^{(s+t-3/2)/(s+t)} \right)
  \\ &= N^{s+t-1}\prod_{j\leqslant s}(\| \hat f_j \|_\infty/N)^{1/(2s+2t)}.
\end{align*}
\end{proof}

\begin{proof}[Proof of Theorem \ref{thm3.4} given Theorem \ref{thm3.8}]

We fix the parameters $\delta,r,L_1,L_2,P$, as we did at the start of \S\ref{sec4}, and allow all forthcoming implicit constants to depend on these parameters. Let $\tilde{\delta}\in(0,1)$ be sufficiently small in terms of these parameters. We also choose $w\in\N$ to be sufficiently large in terms of the fixed parameters, and define $W$ and $D=W^2$ as in \S\ref{sec4}. Let $Z$ and $N$ be defined by (\ref{eqn4.2}).

We begin by addressing the assumption that the quantity $Z$ defined in (\ref{eqn4.2}) is a positive integer, which is equivalent to requiring $D$ to divide $X-r_D$. If this is not the case, then we replace $X$ with $X'=X-m$, where $m\in[D]$ is chosen such that $D$ divides $X'-r_D$. Provided that $X$ is sufficiently large relative to $D$ and $\delta$, we have $(X/2)<X'\leqslant X$, and every $A\subseteq[X]$ with $|A|\geqslant\delta X$ satisfies $|A\cap [X']|\geqslant |A|-D\geqslant (\delta/2)X'$. Hence, by replacing $(X,\delta)$ with $(X',\delta/2)$, we may henceforth assume that $Z\in\N$.

\bigskip

Let $[X]=\cC_1\cup\cdots\cup\cC_r$ and set
\begin{equation*}
    \tilde \cC_i := \{ z \in [Z]: r_D + D z \in C_i \} \quad (1\leqslant i\leqslant r).
\end{equation*}
Let $k \in [r]$ be the index provided by applying Theorem \ref{thm3.8} with respect to the colouring $[Z]=\tilde{\cC}_1 \cup\cdots\cup\tilde{\cC}_r$ and with $\tilde{\delta}$ in place of $\delta$. Our goal is to show that this $k\in[r]$ satisfies the conclusion of Theorem \ref{thm3.4}. In view of the remarks following the statement of Theorem \ref{thm3.8}, and since $D$ is ultimately bounded above in terms of the fixed parameters only, we note that
\begin{equation*}
    |\cC_k|\geqslant |\tilde{\cC}_k| \gg Z \gg X.
\end{equation*}

Let $A\subseteq[X]$ satisfy $|A|\geqslant \delta X$, and define $\kappa,b,\cA$ and $\nu=\nu_b$ as in \S\ref{sec4}. In particular, recall that $b\in[W\kappa]$ is chosen to ensure that Lemma \ref{lem4.1} holds. Let
\[
f = \nu 1_{\cA}, \qquad h_i(n) = \frac{N}{Z}\sum_{\substack{z \in \tilde{\cC}_i \\ P_D(z) = n}} 1 \quad (1\leqslant i\leqslant r).
\]
In light of \eqref{eqn3.5}, the function $h_i$ is supported on $[N]$.
Recalling the Fourier decay estimate (\ref{eqn5.1}), the dense model lemma \cite[Theorem 5.1]{Pre2017} provides a function $g$ such that 
\[
0 \le g \le 1_{[N]},
\quad
\| \hat f - \hat g \|_\infty \ll (\log w)^{-3/2} N.
\]
For $\ell \in [s]$, write
$\bu^{(\ell)} = (u^{(\ell)}_1,\ldots,u^{(\ell)}_s)$,
where
\[
u^{(\ell)}_j = 
\begin{cases}
g, &\text{if } j < \ell \\
f - g, &\text{if } j = \ell \\
f, &\text{if } j > \ell.
\end{cases}
\]
By the telescoping identity and Lemma \ref{lem7.1}, we now have
\begin{align*}
\Phi(f;h_i) - \Phi(g;h_i) 
&= 
\sum_{\ell \le s}
\Phi(\bu^{(\ell)}; h_i) \ll (\log w)^{-3/(4s+4t)} N^{s+t-1} \qquad (1 \le i \le r).
\end{align*}
Recall from Lemma \ref{lem4.1} that
\[
\sum_{n \in \bZ} f(n) \gg N.
\]
As $\hat f(0) - \hat g(0) \ll (\log w)^{-3/2} N$, for $w$ sufficiently large, it follows that
\[
\sum_{n \in \bZ} g(n) \gg N.
\]

Let $c$ be a small, positive constant, which depends only on the fixed parameters, and set
\[
\tilde \cA = \{ n \in \bZ: g(n) \ge c \}.
\]
By the popularity principle (see \cite[Exercise 1.1.4]{TV2006}), we have $|\tilde \cA| \gg N$. In particular, provided $\tilde{\delta}$ is sufficiently small, we can ensure that $|\tilde\cA| \geqslant \tilde{\delta}N$. Thus, Theorem \ref{thm3.8} informs us that
\[
\Phi(\tilde \cA;h_k) \gg N^{s+t-1}.
\]
We therefore have $\Phi(g;h_k) \gg N^{s+t-1}$, whence $\Phi(f;h_k) \gg N^{s+t-1}$, and finally
\begin{align*}
|\{ (\bx,\by) \in A^s \times\cC_k^t: L_1(P(\bx)) = L_2(P(\by)) \}| &
\geqslant \lVert f\rVert_\infty ^{-s}\lVert h_k\rVert_\infty ^{-t}\Phi(f; h_k) \\
&\gg_w (X^{1-d})^s (Z/N)^t N^{s+t-1} \\
&\gg_w X^{s+t-d}.
\end{align*}
Since $w= O_{\delta,r,L_1,L_2,P}(1)$, the proof is complete.
\end{proof}

\section{Arithmetic regularity}\label{sec8}

In this section, we prove Theorem \ref{thm3.8} using the arithmetic regularity lemma. This lemma, originally due to Green \cite{Gre2005B}, allows one to decompose the indicator function $1_\cA$ of a dense set $\cA\subseteq[N]$ as $1_{\cA}=f_{\str}+f_{\sml}+f_{\unf}$, for some `structured' function $f_{\str}:[N]\to[0,1]$ and some `small' functions $f_{\sml},f_{\unf}:[N]\to[-1,1]$. The upshot is that, after some careful analysis, we can count solutions to $L(\bn)=L_2(P_D(\bz))$ with $\bn\in\cA^s$ by instead counting solutions with the $n_i$ weighted by $f_{\str}$. This new counting problem can be addressed directly by exploiting the `almost-periodicity' of the function $f_{\str}$.

One issue with this approach is that Theorem \ref{thm3.8} requires us to find a colour class $\cC_k$ which delivers the conclusion (\ref{eqn3.10}) for all $\delta$-dense sets $\cA\subseteq[N]$ simultaneously. Unfortunately, the arithmetic regularity lemma is not well-suited to decomposing a potentially unbounded collection of indicator functions $1_{\cA}$ in such a way that we obtain a consistent structure for each of the corresponding functions $f_{\str}$. 
Instead, as in the work of Prendiville \cite[\S3]{Pre2021}, we fix an arbitrary finite collection of dense sets $\cA_1,\ldots,\cA_r\subseteq[N]$, which can then be decomposed simultaneously, and find a colour class $\cC_k$ for which (\ref{eqn3.10}) holds for all $\cA\in\{\cA_1,\ldots,\cA_r\}$. This delivers the following variation of Theorem \ref{thm3.8}.

\begin{thm}\label{thm8.1}
Let $d$ and $r$ be positive integers, and let $0<\delta<1$ be a real number. 
Let $P$ be an intersective integer polynomial of degree $d$ which satisfies (\ref{eqn3.5}).
Let $s \ge 1$ and $t \ge 0$ be integers such that $s + t \ge s_0(d)$. Let
$L_1(\bx) \in \bZ[x_1,\ldots,x_s]$
be a non-degenerate linear form for which $\gcd(L_1) = 1$ and $L_1(1,\ldots,1)=0$, and let $L_2(\by) \in \bZ[y_1,\ldots,y_t]$ be a non-degenerate linear form.
Let $D, Z \in \bN$ satisfy $Z \ge Z_0(D, r, \del, L_1, L_2, P)$, and set $N:=P_D(Z)$.
Let
\[
\cA_i \subseteq [N], \quad
|\cA_i| \ge \del N \qquad (1 \le i \le r).
\]
If $[Z] = \cC_1 \cup \cdots \cup \cC_r$, then there exists $k \in [r]$ such that
\begin{equation}\label{eqn8.1}
\# \{ (\bn,\bz) \in \cA_i^s \times \cC_k^t: L_1(\bn) = L_2(P_D(\bz)) \} \gg N^{s-1} Z^t \qquad (1 \le i \le r).
\end{equation}
The implied constant may depend on $L_1, L_2, P, r,$ and  $\del$, but does not depend on $D$.
\end{thm}

Although Theorem \ref{thm8.1} may seem weaker than Theorem \ref{thm3.8}, they are in fact equivalent. This may be proved directly, however, as this argument may be applicable in other contexts, we instead encapsulate the proof strategy in the following combinatorial result, for which we have not found a reference. 
This lemma is essentially a finite version of the axiom of choice.

\begin{lemma}\label{lem8.2}
Let $U,V$ be non-empty sets such that $V$ is finite. Let $E\subseteq U\times V$. Suppose that for every $S\subseteq U$ with $|S|\leqslant |V|$ there exists $v\in V$ such that $(s,v)\in E$ for all $s\in S$. Then there exists $x\in V$ such that $(u,x)\in E$ for all $u\in U$.
\end{lemma}

\begin{proof}
Suppose that for each $x\in V$ there exists $s_x \in U$ such that $(s_x , x)\notin E$. Taking $S=\{s_{x}:x\in V\}$ establishes the contrapositive.
\end{proof}

\begin{proof}[Proof of Theorem \ref{thm3.8} given Theorem \ref{thm8.1}]
In view of Proposition \ref{prop3.10}, it suffices to consider only the case where $\gcd(L_1)=1$. Let $V=[k]$ and set $U=\{\cA\subseteq[N]:|\cA|\geqslant\delta N\}$. Let $E$ denote the set of pairs $(\cA,k)$ such that the inequality (\ref{eqn8.1}) holds for $\cC_k$ with $\cA_i=\cA$. We conclude from Theorem \ref{thm8.1} and Lemma \ref{lem8.2} that Theorem \ref{thm3.8} holds with the implicit constant in (\ref{eqn3.10}) equal to the one in (\ref{eqn8.1}), and with the same $Z_{0}(D,r,\delta,L_1,L_2,P)$.
\end{proof}

\subsection{The arithmetic regularity lemma}

We now introduce the version of the arithmetic regularity lemma that we use to prove Theorem \ref{thm8.1}.
In the sequel, we write $\bT^K$ for the $K$-dimensional torus $(\bR/\bZ)^K$. This is equipped with a metric $(\balpha,\bbeta)\mapsto\lVert \balpha - \bbeta\rVert$, where
\begin{equation*}
     \lVert \btheta\rVert := \max_{1\leqslant i\leqslant K}\min_{n\in\bZ}|\theta_i - n| \qquad (\btheta=(\theta_1,\ldots,\theta_K) \in\T^K).
\end{equation*}
This allows us to define \emph{Lipschitz functions} on $\T^K$. Given a positive real number $H$, a function $F:\T^{K}\to\R$ is \emph{$H$-Lipschitz} if 
\begin{equation*}
    |F(\balpha) - F(\bbeta)| \leqslant H\lVert \balpha - \bbeta\rVert \qquad (\balpha,\bbeta \in\T^K).
\end{equation*}

\begin{lemma}[Arithmetic regularity lemma]\label{lem8.3}
	Let $r\in\N$, $\sig>0$, and let $\cF:\R_{\geqslant 0}\to\R_{\geqslant 0}$ be a monotone increasing function. Then there exists a positive integer $K_{0}(r;\sig,\cF)\in\N$ such that the following is true. Let $N\in\N$ and $f_{1},\ldots,f_r:[N]\to[0,1]$. Then there is a positive integer $K\leqslant K_{0}(r;\sig,\cF)$ and a phase $\btheta\in\T^{K}$ such that, for every $i\in[r]$, there is a decomposition
	\begin{equation*}
	f_{i}=f_{\str}^{(i)}+f_{\sml}^{(i)}+f_{\unf}^{(i)}
	\end{equation*}
	of $f_{i}$ into functions $f_{\str}^{(i)},f_{\sml}^{(i)},f_{\unf}^{(i)}:[N]\to[-1,1]$ with the following stipulations.
	\begin{enumerate}[\upshape(I)]
		\item\label{itemNon} The functions $f_{\str}^{(i)}$ and $f_{\str}^{(i)}+f_{\sml}^{(i)}$ take values in $[0,1]$.
		\item The function $f_{\sml}^{(i)}$ obeys the bound $\lVert f_{\sml}^{(i)}\rVert_{L^{2}(\Z)}\leqslant\sig\lVert 1_{[N]}\rVert_{L^{2}(\Z)}$.
		\item The function $f_{\unf}^{(i)}$ obeys the bound $\lVert \hat{f}_{\unf}^{(i)}\rVert_{\infty}\leqslant\lVert \hat{1}_{[N]}\rVert_{\infty}/\cF(K)$.
		\item\label{itemSum} The function $f_{\str}^{(i)}$ satisfies $\sum_{m=1}^{N}(f_{i}-f_{\str}^{(i)})(m)=0$.
		\item\label{itemStr} There exists a $K$-Lipschitz function $F_{i}:\T^{K}\to[0,1]$ such that $F_{i}(x\btheta )=f_{\str}^{(i)}(x)$ for all $x\in[N]$. 
	\end{enumerate}
\end{lemma}

\begin{proof}
	This is essentially \cite[Lemma 3.3]{Pre2021} and can be proved using the methods of \cite[Theorem 1.2.11]{Tao2012} or \cite[Theorem 5]{Ebe2016} (see also \cite[Lemma 3]{Sal2020}). 
	
	For the convenience of the reader, with reference to the arguments and notation of \cite{Tao2012}, we outline the minor modifications one needs to make to obtain the required result. Let $\cF_{0}:\R_{\geqslant 0}\to\R_{\geqslant 0}$ be defined by $\cF_{0}(x):=\cF(rx)$. By following the iterative procedure given in the proof of \cite[Theorem 1.2.11]{Tao2012} (with $F$ replaced by $\cF_0$), one obtains a sequence of factors $\cB_1^{(i)}\subset \cB_2^{(i)}\subset\ldots$ for each $i\in[r]$. The energies
$\lVert\mathbf{E}(f|\cB^{(i)}_{1})\rVert_{L^{2}([N])}^2$,
$\lVert\mathbf{E}(f|\cB^{(i)}_{2})\rVert_{L^{2}([N])}^2, \ldots$
are monotone increasing between $0$ and $1$, so it follows from
the pigeonhole principle that there exists $k\ll r\sig^{-2}$ such that
	\begin{equation*}
	    \max_{1\leqslant i\leqslant r}\left( \lVert\mathbf{E}(f|\cB^{(i)}_{k+1})\rVert_{L^{2}([N])}^2 -\lVert\mathbf{E}(f|\cB^{(i)}_{k})\rVert_{L^{2}([N])}^2\right)\leqslant \sig^{2}.
	\end{equation*}
	The choice of factors $\cB_j^{(i)}$ delivered by this argument then shows that, upon setting
	\begin{equation*}
	    f_{\str}^{(i)}:=\mathbf{E}(f_i|\cB_{k}),\quad
	    f_{\sml}^{(i)}:=\mathbf{E}(f_i|\cB_{k+1})-\mathbf{E}(f_i|\cB_{k}),\quad f_{\unf}^{(i)}:=f_i-\mathbf{E}(f_i|\cB_{k+1}),
	\end{equation*}
	properties (\ref{itemNon})-(\ref{itemSum}) hold with $\cF_0$ in place of $\cF$. We also have (\ref{itemStr}) but with some $\btheta^{(i)}\in\T^{K}$ for each $i\in[r]$ in place of the desired $\btheta$.
	To establish (\ref{itemStr}) in the form given above, we set $\btheta:=(\btheta^{(1)},\ldots,\btheta^{(r)})\in\T^{Kr}$. Thus, for each $i\in[r]$, we can define a projection map $\pi_{i}:\T^{Kr}\to \T^{K}$ such that $\pi_{i}(\btheta)=\btheta^{(i)}$, whence $f_{\str}^{(i)}(x)=F_{i}\circ\pi_{i}(x\btheta)$ for all $x\in[N]$. Since each $F_{i}\circ\pi_{i}$ is $Kr$-Lipschitz, and since $\cF_0(K)=\cF(Kr)$, we may replace $K$ with $Kr$ to complete the proof.
\end{proof}

To prove Theorem \ref{thm8.1}, we apply the arithmetic regularity lemma above to decompose the indicator functions $1_{\cA_i}$ of our dense sets $\cA_i\subseteq[N]$. As in \S\ref{sec5} and \S\ref{sec6}, where we focused our attention on a single weight function $\nu=\nu_b$, it is convenient for us to first study the consequences of applying the arithmetic regularity lemma to a single function $f$. In such instances, we omit the index $i$ and write $f=f_{\str}+f_{\sml}+f_{\unf}$ for the decomposition provided by Lemma~\ref{lem8.3}. One can think of these results as pertaining to $f=1_{\cA_i}$ for some $i\in[r]$, with the resulting conclusions being uniform in $i$. 

\bigskip

Given finitely supported functions $f_{1},\ldots,f_{s},g_{1},\ldots,g_{t}:\Z\to\bC$, define the counting operator
\begin{equation*}
\Lambda_{D}(f_{1},\ldots,f_{s};g_{1},\ldots,g_{t}):=\sum_{L_1(\bn)=L_{2}(P_D(\bz)) }f_{1}(n_{1})\cdots f_{s}(n_{s})g_{1}(z_{1})\cdots g_{t}(z_{t}).
\end{equation*}
As with the counting operator $\Phi$, we make use of the abbreviations
\[
\Lambda_D(f_1,\ldots, f_s;h) := \Lambda_D(f_1,\ldots,f_s;h,\ldots,h),
\qquad
\Lambda_D(f;h) := \Lambda_D(f,\ldots,f;h,\ldots,h),
\]
and, for finite $A, B \subset \bZ$:
\[
\Lambda_D(f_1,\ldots,f_s;B) := \Lambda_D(f_1,\ldots,f_s;1_B),
\qquad
\Lambda_D(A;B) := \Lambda_D(1_A;1_B).
\]

By a change of variables, one can relate $\Lambda_D$ to the counting operator $\Phi$ defined by (\ref{eqn7.1}) which we studied in \S\ref{sec7}. In particular, one can adapt Lemma \ref{lem7.1} to $\Lambda_D$ as follows.

\begin{lemma}[Fourier control]\label{lem8.4}
    Let $f_{1},\ldots,f_{s},g_{1},\ldots,g_{t}:\Z\to\bR$ be functions supported on $[N]$. Then for any $B\subseteq[Z]$, where $N=P_D(Z)$, we have
    \begin{equation*}
        |\Lambda_{D}(f_{1},\ldots,f_{s};B) -\Lambda_{D}(g_{1},\ldots,g_{s};B)| \ll
        \max_{1\leqslant i\leqslant s}(\| \hat f_{i}-\hat g_{i} \|_\infty/N)^{1/(2s+2t)}
        N^{s-1}Z^{t}.
    \end{equation*}
\end{lemma}
\begin{proof}
   Define the function $h:\Z\to\bR$ by
    \begin{equation*}
        h(x):=
        \begin{cases}
            1_{B}(z),\quad &\text{if there exists }z\in[Z]\text{ such that }x=P_{D}(z) \\
            0, &\text{otherwise}.
        \end{cases}
    \end{equation*}
    Now note that, for all finitely-supported $F_{1},\ldots,F_{s}:\Z\to\bR$, we have
    \begin{equation*}
        \Lambda_{D}(F_{1},\ldots,F_{s};B)=\Phi(F_{1},\ldots,F_{s};h).
    \end{equation*}
    Let $\mu_D$ be given by (\ref{eqn6.1}).
    Since $|h|\leqslant (N^{-1}Z)\mu_{D}$, we deduce from the telescoping identity and Lemma \ref{lem7.1}, as in \S\ref{sec7}, that
    \begin{align*}
        |\Phi(f_{1},\ldots,f_{s};(NZ^{-1})h) - \Phi(g_{1},\ldots,g_{s};(NZ^{-1})h)| \ll \max_{1\leqslant i\leqslant s}(\| \hat f_{i}-\hat g_{i} \|_\infty/N)^{1/(2s+2t)}
        N^{s+t-1}.
    \end{align*}
    Here we have used the trivial bound $\lVert \hat{f}_i\rVert_{\infty}\leqslant N$ for all $i\in[s]$.
    Multiplying both sides by $(N^{-1}Z)^t$ completes the proof.
\end{proof}

An immediate consequence of this result is that we can show that $\Lambda_D(f;B)$ is well-approximated by $\Lambda(f_\str +f_\sml;B)$.

\begin{lemma}[Removing $f_{\unf}$]\label{lem8.5}
	Let $f:\Z\to[0,1]$ be supported on $[N]$. Let $\sig>0$, and let $\cF:\R_{\geqslant 0}\to\R_{\geqslant 0}$ be a monotone increasing function. Let $f_{\str},f_{\sml},f_{\unf}$ be the functions provided by applying Lemma \ref{lem8.3} to $f$. Then for any $B\subseteq[Z]$, we have
	\begin{equation*}
	\lvert\Lambda_{D}(f;B)-\Lambda_{D}(f_{\str}+f_{\sml};B)\rvert\ll_{P}N^{s-1}Z^{t}\cF(K)^{-1/(2s+2t)}.
	\end{equation*}
\end{lemma}
\begin{proof}
This follows immediately from Lemmas \ref{lem8.3} and \ref{lem8.4} with $f_{i}=f$ and $g_{i}=f_{\str}+f_{\sml}$ for all $i\in[s]$.
\end{proof}

\subsection{Polynomial Bohr sets}

Having removed $f_\unf$, it remains to obtain a lower bound for the quantity $\Lambda(f_\str +f_\sml;B)$, thereby producing a lower bound for $\Lambda(f;B)$. As in typical applications of the arithmetic regularity lemma, this is accomplished by exploiting the `almost-periodicity' of the function $f_\str$. Explicitly, this is the observation that, as $F$ is a Lipschitz function, we have $f_\str(n+d)\approx f_\str(n)$ whenever $n,n+d\in[N]$ are such that $\lVert d\btheta\rVert$ is small. The set of such $d$ is known as a \emph{Bohr set}. Since we are interested in the case where $d=P_D(z)$ for some $z\in[Z]$, we therefore need to consider \emph{polynomial Bohr sets}, which are defined as follows.

\begin{defn}[Bohr sets]
	Let $K\in\N$, $\rho>0$, and $\balpha\in\T^{K}$. Let $Q\in\Z[X]$ be an integer polynomial of positive degree. The (\emph{polynomial}) \emph{Bohr set} $\bohr_{Q}(\balpha,\rho)$ is the set
	\begin{equation*}
		\bohr_{Q}(\balpha,\rho):=\{n\in\N:\lVert Q(n)\balpha\rVert<\rho\}=\bigcap_{i=1}^{K}\{n\in\N: \lVert Q(n)\alpha_{i}\rVert<\rho\}.
	\end{equation*}
\end{defn}

Bohr sets are well-studied objects in additive combinatorics and analytic number theory \cite[\S4.4]{TV2006}.
In the classical setting $Q(n) = n$, it is well known that the Bohr set has positive lower density.
For our applications, we only need to ensure that
\begin{equation*}
    |\bohr_{Q}(\balpha,\rho)\cap[Z]|\gg_{d,K,\rho} Z
\end{equation*}
for $Z$ large enough relative to $Q,K,\rho$,
where $Q$ is an intersective polynomial of degree $d$. The crucial aspect of this bound which we emphasise is that the implicit constant does not depend on the frequency $\balpha$ nor on the coefficients of $Q$. Note that intersectivity is necessary even to ensure that the Bohr set is non-empty, for otherwise there is a local obstruction.

\bigskip

We start with the case $Q(0)=0$, which was investigated in \cite{Cha2022}.
\begin{lemma}\label{lem8.7}
Let $K\in\N$, $\rho>0$ and $\balpha\in\T^{K}$. Let $Q\in\Z[X]$ be a polynomial of degree $d\in\N$ such that $Q(0)=0$. Then there exists a positive real number $\Delta_{0}(\rho)=\Delta_{0}(d, K, \rho)$ and a positive integer $Z_{0}(d,K,\rho)$ such that
if $Z\geqslant Z_{0}(d,K,\rho)$
then
\begin{equation*}
    |\bohr_{Q}(\balpha,\rho)\cap[Z]|\geqslant \Delta_{0}(\rho)Z.
\end{equation*}

\end{lemma}

\begin{proof}
Write $Q(X)=\sum_{i=1}^{d}a_{i}X^{i}$ for some $a_{1},\ldots,a_{d}\in\Z$. We abuse notation and write $\bohr_{i}(\balpha,\rho)$ for $\bohr_{P}(\balpha,\rho)$ when $P(X)=X^{i}$. The triangle inequality implies that
\begin{equation*}
    \bohr_{Q}(\balpha,\rho)\supseteq \bigcap_{i=1}^{d}\bohr_{i}(\bbeta^{(i)},\rho/d),
\end{equation*}
where $\bbeta^{(i)}:=a_{i}\balpha$.
From $d$ applications of \cite[Corollary 6.9]{Cha2022}, we deduce that there exists a positive integer $M\ll_{d,K,\rho} 1$ such that $\{x,2x,\ldots,Mx\}\cap\bohr_{Q}(\balpha,\rho)\neq\emptyset$ holds for all $x\in\N$. Thus, we conclude from \cite[Lemma 4.2]{CLP2021} that $|\bohr_{Q}(\balpha,\rho)\cap[Z]|\geqslant Z/(2M^{2})$ holds for all $Z\geqslant M$.
\end{proof}

We now consider the general case where $Q$ is an arbitrary intersective polynomial. To deduce the required result from Lemma \ref{lem8.7}, we need to know that
\[
\sup_{\balp \in \bT^K}
\min_{z \in [Z]}
\lVert Q(z)\balpha\rVert
\to 0 \qquad (Z \to \infty).
\]
As previously mentioned, the significant feature is uniformity in $\balpha$. Such a result follows from the much stronger quantitative bound given in \cite[Theorem 1]{LR2015}. Using this, we now establish a lower bound for the density of an arbitrary intersective polynomial Bohr set.

\begin{lemma}\label{lem8.8}
Let $K\in\N$, $\rho>0$ and $\balpha\in\T^{K}$. Let $Q\in\Z[X]$ be an intersective polynomial of degree $d\in\N$. Then there exists a positive real number $\Delta_1(\rho)=\Delta_1(d,K;\rho)$ and a positive integer $Z_{1}(Q,K,\rho)$ such that
if $Z\geqslant Z_{1}(Q,K,\rho)$ 
then
\begin{equation*}
    |\bohr_{Q}(\balpha,\rho)\cap[Z]|\geqslant \Delta_1(\rho)Z.
\end{equation*}
\end{lemma}

\begin{proof}
If $Z$ is sufficiently large in terms of $(Q,K,\rho)$, then it follows from \cite[Theorem 1]{LR2015} that there exists $t\in\bohr_{Q}(\balpha,\rho/2)$ with $t<Z/2$. Let $P(X):=Q(X+t)-Q(t)$. Since $P(0)=0$, Lemma \ref{lem8.7} ensures that
\begin{equation*}
    |\bohr_{P}(\balp,\rho/2)\cap[Z/2]|\gg_{d,K,\rho} Z.
\end{equation*}
By the triangle inequality, we now have
\begin{equation*}
    \left\lbrace t+x: x\in (\bohr_{P}(\balpha,\rho/2)\cap[Z/2])\right\rbrace\subseteq \bohr_{Q}(\balpha,\rho)\cap[Z],
\end{equation*}
from which the desired bound follows.
\end{proof}

\subsection{Completing the proof of Theorem \ref{thm8.1}}

Recall that the coefficients of $L_1$ are coprime. This implies that there exists $\bv\in\Z^{s}$ whose entries have size $O_{L_1}(1)$ such that $L_1(\bv)=1$.
Thus, for any finitely supported
$f_{1},\ldots,f_{s},g_{1},\ldots,g_{t}:\Z\to\bC$, we may write
\begin{equation*}
\Lambda_{D}(f_{1},\ldots,f_{s};g_{1},\ldots,g_{t})=\sum_{{\bz\in\Z^{t}}}g_{1}(z_{1})\cdots g_{t}(z_{t})\Psi_{\bz}(f_{1},\ldots,f_{s}),
\end{equation*}
where
\[
\Psi_{\bz}(f_{1},\ldots,f_{s}):=\sum_{L_1(\bn)=0}\prod_{i=1}^{s}f_{i}(n_{i}+v_{i}L_2(P_D(\bz))).
\]
For brevity, we write $\Psi_{\bz}(f):=\Psi_{\bz}(f,\ldots,f)$. Following \cite[\S6.1]{Cha2022}, we proceed to study these auxiliary counting operators $\Psi_{\bz}$, with a view towards obtaining a lower bound for $\Lambda_{D}$ by summing over $\bz$ lying in a polynomial Bohr set.

\begin{lemma}[Generalised von Neumann for $\Psi$]\label{lem8.9}
	Let $\bz\in\Z^{t}$, and let $\Psi_{\bz}$ be defined as above. If $f,g:[N]\to[0,1]$, then
	\begin{equation*}
	\lvert\Psi_{\bz}(f)-\Psi_{\bz}(g)\rvert\leqslant s N^{s-1}(\lVert f-g\rVert_{2}^{2}/N)^{1/2}.
	\end{equation*}
\end{lemma}

\begin{proof}
For all $f_{1},\ldots,f_{s}:\bZ\to[-1,1]$ supported on $[N]$, we proceed to show that
\begin{equation*}
	\lvert\Psi_{\bz}(f_{1},\ldots,f_{s})\rvert\leqslant (\lVert f_{j}\rVert_{2}^{2}/N)^{1/2}N^{s-1}
	\end{equation*}
	for all $j\in [s]$.
Once this is established, the lemma then follows from the telescoping identity
\begin{equation*}
    \Psi_{\bz}(f) - \Psi_{\bz}(g) =
    \sum_{i=1}^{s}\Psi_\bz(h_{1},\ldots,h_{i-1},h_{i}-g_{i},g_{i+1},\ldots,g_{s}),
\end{equation*}
where $h_{i}=f$ and $g_{i}=g$ for all $i\in[s]$.

We demonstrate only the case $j=s$, as the other cases follow by symmetry. Given $\bn=(n_{1},\ldots,n_{s})\in\Z^{s}$, we write $L_1(\bn)=L_1(\widetilde{\bn},n_{s})$, where $\widetilde{\bn}=(n_{1},\ldots,n_{s-1})\in\Z^{s-1}$.
Let $u=L_1(0,\ldots,0,v_s L_2(P_D(\bz)))\in\Z$. By the change of variables 
$n=n_{s}+v_s L_2(P_D(\bz))$, we have
\begin{equation*}
    \Psi_{\bz}(f_{1},\ldots,f_{s}) = \sum_{n\in\Z}f_{s}(n)\sum_{\substack{\widetilde{\bn}\in\Z^{s-1}\\ L_1(\widetilde{\bn},n)=u}}\prod_{i=1}^{s-1}f_{i}(n_{i}+v_{i}L_2(P_D(\bz))).
\end{equation*}
Note that $f_s(n)$ vanishes if $n\notin[N]$. Hence, by applying Cauchy--Schwarz to the outer sum over $n$, we deduce that
\begin{equation*}
    |\Psi_{\bz}(f_{1},\ldots,f_{s})|^{2} \leqslant \lVert f_{s}\rVert_{2}^{2}\sum_{n=1}^{N}\left(\sum_{\substack{\widetilde{\bn}\in\Z^{s-1}\\ L_1(\widetilde{\bn},n)=u}}\prod_{i=1}^{s-1}f_{i}(n_{i}+v_{i}L_2(P_D(\bz)))\right)^{2}.
\end{equation*}
Since $|f_{i}|\leqslant 1_{[N]}$ for all $i$, we deduce that the inner sum over $\widetilde{\bn}$ is bounded above by
\begin{equation*}
\# \{\widetilde{\bn}\in\Z^{s-1}:L_1(\widetilde{\bn},n)=u, \qquad
(n_{i}+v_{i}L_2 (P_D (\bz)))\in[N] \quad (i\in[s-1]) \}
\leqslant N^{s-2}.
\end{equation*}
Inserting this bound reveals that
\begin{equation*}
    |\Psi_{\bz}(f_{1},\ldots,f_{s})|^{2} \leqslant \lVert f_{s}\rVert_{2}^{2}\sum_{n=1}^{N}N^{2(s-2)}=(\lVert f_{i}\rVert_{2}^{2}/N)N^{2(s-1)},
\end{equation*}
and taking square roots completes the proof.
\end{proof}

Before we use this lemma to obtain a lower bound for $\Psi_{\bz}(f_{\str}+f_{\sml})$, we require two additional lemmas. Firstly, we require a functional version of the supersaturation result of Frankl, Graham, and R\"{o}dl \cite[Theorem 2]{FGR1988} for density regular linear equations.

\begin{lemma}\label{lem8.10}
    Let $\delta>0$, and let $f:[N]\to[0,1]$. If $\lVert f\rVert_{1}\geqslant \delta N$, then
    \begin{equation}\label{eqn8.2}
        \sum_{L_1(\bn)=0}f(n_{1})\cdots f(n_s) \gg_{L_1,\delta} N^{s-1}.
    \end{equation}
\end{lemma}

\begin{proof}
    Let $\Omega=\{x\in[N]:f(x)\geqslant \delta/2\}$. The popularity principle \cite[Exercise 1.1.4]{TV2006} implies that $|\Omega|\geqslant (\delta/2)N$, and so
    \begin{equation*}
        \sum_{L_1(\bn)=0}\prod_{i=1}^{s}f(n_{i})\geqslant 
        \sum_{L_1(\bn)=0}\prod_{i=1}^{s}\left((\delta/2)1_{\Omega}(n_{i})\right)=(\delta/2)^{s}|\{\bn\in\Omega^{s}: L_1(\bn)=0\}|.
    \end{equation*}
    Since $L_1(1,\ldots,1)=0$, the required bound now follows from \cite[Theorem 2]{FGR1988}.
\end{proof}

\begin{remark}
Alternatively, one can prove Lemma \ref{lem8.10} without using \cite[Theorem 2]{FGR1988}. After applying the arithmetic regularity lemma (Lemma \ref{lem8.3}), one can then show that the sum (\ref{eqn8.2}) for $f_\str$ is $\gg_{L_1,\delta}N^{s-1}$ by restricting to a sum over $n_1,\ldots,n_s$ lying in a linear Bohr set (see \cite[\S6]{Cha2022} for further details).
\end{remark}

As mentioned previously, we intend to make use of the almost periodicity of $f_\str$ to obtain a lower bound for $\Psi_{\bz}(f_{\str}+f_{\sml})$ when $\bz$ lies in an intersective polynomial Bohr set. However, we have to be conscious of the fact that we are relying on the relation $f_{\str}(n)=F(n\btheta)$, which only holds for $n\in[N]$. To guarantee that quantities of the form $n_{i} + v_{i}L_{2}(P_{D}(z_{j}))$ lie in $[N]$, we restrict our variables according to $(\bn,\bz)\in[c(\eta)N,(1-c(\eta)N]^s\times [\eta Z]^t$, for some sufficiently small $\eta>0$ and some corresponding quantity $c(\eta)>0$ such that $c(\eta)\to 0^{+}$ as $\eta\to 0^{+}$. Moreover, since our final bound (\ref{eqn3.10}) does not depend on $D$, we need to ensure that the decay rate of $c(\eta)$ is independent of $D$. This is accomplished by the following simple lemma on polynomial growth.

\begin{lemma}\label{lem8.12}
Let $P$ be an intersective integer polynomial of degree $d\in\N$ satisfying (\ref{eqn3.5}). Then there exists $M_0(P) \in \bN$ such that the following is true. Let $\eta\in(0,1)$, let $D\in\N$, and define the auxiliary polynomial $P_D$ by (\ref{eqn3.9}). If 
$M\geqslant (M_0(P) + 1)/\eta$, then 
\[
P_D(\eta M)\leqslant (4\eta)^d P_D(M).
\]
\end{lemma}

\begin{proof}
Let $\ell_P$ denote the leading coefficient of $P$. Since (\ref{eqn3.5}) holds, we know that $\ell_P\geqslant 1$, and that there exists a positive integer $M_0(P)\geqslant 4$ such that
\begin{equation*}
    \ell_P Y^d \leqslant 2P(Y) \leqslant 3\ell_P Y^d
\end{equation*}
holds for all real $Y\geqslant M_0(P) $. Since $-D<r_D\leqslant 0$, it follows that if $M\geqslant (M_0(P) + 1)/\eta$ then
\begin{equation*}
    \frac{P_D(\eta M)}{P_D(M)} 
    \leqslant \frac{3(r_D + D\eta M)^d}{(r_D + D M)^d}
    \leqslant \frac{3(D\eta M)^d}{(D M - D)^d}
    = 3\eta^d\left(1 + \frac{1}{M-1}\right)^d.
\end{equation*}
The asserted bound now follows upon noting that $M\geqslant M_0(P)\geqslant 4$.
\end{proof}

\begin{lemma}[Lower bound for $\Psi_{\bz}(f_{\str}+f_{\sml})$]\label{lem8.13}
For all $\delta>0$, there exist positive constants $c_{1}(\delta)=c_{1}(L_1,L_2;\delta)>0$ and $\eta=\eta(d,L_1,L_2,\delta)>0$ such that the following is true. Suppose $f:\Z\to[0,1]$ is supported on $[N]$ and satisfies $\lVert f\rVert_1\geqslant \delta N$. Given $\sig>0$ and a monotone increasing function $\cF:\R_{\geqslant 0}\to\R_{\geqslant 0}$, let $f_{\str}$, $f_{\sml}$, $K$ and $\btheta$ be as given by applying Lemma~\ref{lem8.3} to $f$. Let $\rho>0$ satisfy $K\rho\leqslant 1$, and let $\bz\in\bohr_{P_D}(\btheta,\rho)^{t}$. If $\bz\in[\eta Z]^{t}$, then 
\begin{equation*}
    \Psi_{\bz}(f_{\str}+f_{\sml}) \geqslant \left(c_{1}(\delta)- O_{L_1,L_2}(\sig + K\rho )\right)N^{s-1}. 
\end{equation*}
\end{lemma}

\begin{proof}
Lemma \ref{lem8.9} informs us that
   \begin{equation*}
       \Psi_{\bz}(f_{\str}+f_{\sml}) = \Psi_{\bz}(f_{\str}) + O(\sig N^{s-1}).
   \end{equation*}
   It therefore only remains to estimate $\Psi_{\bz}(f_{\str})$. 
   For each $\bn\in\Z^{s}$, define
   \begin{equation*}
       I_{\bz}(\bn):=
       \begin{cases}
           1, \quad &\text{if }(n_{i}+v_{i}L_2(P_D(\bz)))\in[N]\text{ for all }i\in [s];\\
           0, &\text{otherwise}.
       \end{cases}
   \end{equation*}
Since $\bz\in\bohr_{P_D}(\btheta,\rho)^{t}$, we deduce from property (\ref{itemStr}) of Lemma \ref{lem8.3} that if $\bn\in[N]^s$, then
   \begin{equation*}
       I_{\bz}(\bn)|f_{\str}(n_{i})-f_{\str}(n_{i}+v_{i}L_2(P_D(\bz)))|\ll_{\bv,L_2}K\rho\quad (1\leqslant i\leqslant s).
   \end{equation*}
Thus, by using property (\ref{itemNon}) to bound $f_{\str}$ by $1$, we find that
   \begin{align*}
       \Psi_{\bz}(f_{\str})
      & =\sum_{L_1(\bn)=0}I_{\bz}(\bn)\prod_{i=1}^{s}[f_{\str}(n_{i})+O_{\bv,L_2}(K\rho)]\\
       &=\left(\sum_{L_1(\bn)=0}I_{\bz}(\bn)f_{\str}(n_{1})\cdots f_{\str}(n_{s})\right) + O_{\bv,L_2}(K\rho N^{s-1})
       .
   \end{align*}
   In view of (\ref{eqn3.5}) and Lemma \ref{lem8.12}, we see that
   \begin{equation*}
       |L_2(P_{D}(\bz))|\ll_{L_2} P_{D}(\eta Z) \leqslant (4\eta)^{d}N \quad (\bz\in[\eta Z]^t).
   \end{equation*}
   It follows that there exists a constant $c=c(L_1,L_2,d)>0$ such that, for all $\bz\in[\eta Z]^{t}$, the function $I_{\bz}$ is non-zero on the set $\Omega^{t}$, where $\Omega:=\left(c\eta^d N,(1-c\eta^d)N\right]\cap \bZ$.
We therefore find that \begin{equation*}
\Psi_{\bz}(f_{\str})
\geqslant\left(\sum_{L_1(\bn)=0}g(n_{1})\cdots g(n_{s})\right) - O_{\bv,L_2}(K\rho N^{s-1}),
\end{equation*}
where $g(n):=1_{\Omega}(n)f_{\str}(n)$.
   
Finally, we infer from property (\ref{itemSum}) of Lemma \ref{lem8.3} that $g(1)+\cdots +g(N)\geqslant (\delta-2c\eta^d) N$. Thus, by taking $\eta$ sufficiently small, we can apply Lemma \ref{lem8.10} to $g$ to obtain the required bound.
\end{proof}

Combining all of these results finally allows us to prove Theorem \ref{thm8.1}, thereby completing the proof Theorem \ref{thm3.8}.

\begin{proof}[Proof of Theorem \ref{thm8.1}]
Fix $r\in\N$ and $\delta\in(0,1)$. Let $c_1 (\delta)$ and $\eta=\eta(\delta,P,L_1,L_2)$ be as given in Lemma \ref{lem8.13}. Notice that the conclusion of Lemma \ref{lem8.13} allows us to assume that $c_1 (\delta)<1$, which we do. Let $\sig =c_{1}(\delta)/M$, where $M=M(L_1,L_2)$ is some suitably large positive integer, and let $\Delta_1$ be a function given by Lemma \ref{lem8.8}. Let $\cF:\R_{\geqslant 0}\to\R_{\geqslant 0}$ be a monotone increasing function which satisfies
    \begin{equation}\label{eqn8.3}
\cF(x)^{-1/(2s+2t)} \leqslant \tau c_{1}(\delta)\left(\eta r^{-1}\Delta_{1}(d,x;x^{-1}\sig)\right)^{t}
\end{equation}
    for all $x\in\N$, where $\tau=\tau(P) > 0$ will be chosen shortly.
    
Let $N,Z\in\N$ be as given in the statement of Theorem \ref{thm8.1}, and assume they are sufficiently large in terms of $(\delta,r,L,L_2,P)$. Suppose we have an $r$-colouring $[Z]=\cC_1 \cup\cdots\cup \cC_r$, and sets $\cA_1 ,\ldots,\cA_r\subseteq[N]$ satisfying $|\cA_i|\geqslant\delta N$ for each $i\in[r]$. Applying Lemma \ref{lem8.3} to each of the functions $f_i:=1_{\cA_i}$ provides a decomposition $f_i=f_{\str}^{(i)}+f_{\sml}^{(i)}+f_{\unf}^{(i)}$, as well as associated parameters $K \leqslant K_{0}(r;\sig,\cF)$ and $\btheta\in\T^{K}$. Let $\rho>0$ be defined by the equality $K\rho=\sig$. By our choices of parameters in the previous paragraph, we can assume that $N$ and $Z$ are also sufficiently large relative to $(\sig,\cF,\rho,\eta)$. 
    
Now let $\cC_{i}'=\cC_i \cap [\eta Z]$ for all $i\in[r]$. Recall from Lemma \ref{lem3.7} that $P_D$ is an intersective polynomial of degree $d$. Thus, applying the pigeonhole principle and Lemma \ref{lem8.8} to the colouring $[\eta Z]=\cC_1' \cup\cdots\cup \cC_r '$ yields an index $k\in[r]$ such that 
    \begin{equation}\label{eqn8.4}
|\bohr_{P_D}(\btheta,\rho)\cap \cC_{k}'|\geqslant r^{-1}\Delta_1(d,K;\rho)\eta Z.
    \end{equation}
    It now remains to establish (\ref{eqn8.1}) for this choice of $k$.
    
Let $B:=\cC'_{k} \cap
\bohr_{P_{D}}(\btheta,\rho)$.
For any $\bz\in B^{t}$, if $M$ is large enough, then Lemma~\ref{lem8.13} implies that
    \begin{equation*}
        2\Psi_{\bz}(f_{\str}^{(i)}+f_{\sml}^{(i)})\geqslant c_{1}(\delta)N^{s-1} \quad (1\leqslant i\leqslant r).
    \end{equation*}
Summing over $\bz$ yields
\[
2\Lam_D
(f_{\str}^{(i)} +
f_{\sml}^{(i)}; B)
\ge c_1(\del) |B|^t N^{s-1}.
\]
    Incorporating Lemma \ref{lem8.5} and (\ref{eqn8.4}) reveals that
\begin{align*}
2\Lambda_{D}(\cA_i;B)&\ge
\left(c_{1}(\delta)|B|^{t} - C\cF(K)^{-1/(2s+2t)}Z^{t}\right)N^{s-1}
\\
&\geqslant \left(c_{1}(\delta)r^{-t}
\Delta_1(d,K;K^{-1}\sig)^{t}\eta^{t} - C\cF(K)^{-1/(2s+2t)}\right)N^{s-1}Z^{t},
    \end{align*}
    for all $i\in[r]$ and some constant $C=C(P)>1$. Setting $\tau^{-1}=2C$ in (\ref{eqn8.3}) now gives
    \begin{equation*}
        \Lambda_{D}(\cA_i;\cC_{k})\geqslant \Lambda_{D}(\cA_i;B)\gg_{\delta,r,P,L_1,L_2} N^{s-1}Z^{t} \qquad   (1\leqslant i\leqslant r),
    \end{equation*}
    as required.
\end{proof}

\appendix

\section{Polynomial congruences}

\begin{lemma} \label{lemA.1}
Let $p$ be prime, and let $s \in \bN$. Let $f(X) \in \bZ_p[X]$ have degree $d \in \bN$ and discriminant $\Del$. Assume that $p \nmid c \Del$, where $c$ is the leading coefficient of $f$. Then
\[
\# \{ x \in [p^s]: f(x) \equiv 0 \mmod{p^{s}} \} \le d.
\]
\end{lemma}

\begin{proof} 
Define
\begin{equation} \label{eqnA.1}
X = \{ x \in \bQ_p: f(x) = 0 \}, \qquad 
X_s = \{ x \in [p^s]: f(x) \equiv 0 \mmod{p^{s}} \}.
\end{equation}
Writing $\bar \Del$ for the discriminant of the image $\bar f$ of $f$ in $\bF_p$, we have
\[
p \nmid \bar \Del = \prod_{x \in \bF_p: \bar f(x) = 0} \bar f'(x).
\]
Consequently, if $a \in X_s$ then $p \nmid f'(a)$, so by Hensel's lemma there exists $\tilde a \in X$ such that
\[
\tilde a \equiv a \mmod{p^{s}}.
\]
Hence
\[
|X_s| \le |X| \le d.
\]
\end{proof}

\begin{lemma} \label{lemA.2} Let $p$ be prime, and let $f(X) \in \bZ_p[X]$ be squarefree of degree $d \ge 2$. Let $R \ne 0$ be the resultant of $f$ and $f'$, and let $j > \ord_p(R)$ be an integer. Then
\begin{equation} \label{eqnA.2}
Y_j := \{ y \in \bZ_p: f(y) \equiv f'(y) \equiv 0 \mmod{p^{j}} \}
\end{equation}
is empty.
\end{lemma}

\begin{proof} 
By \cite[Chapter 3, \S 6, Proposition 5]{CLO2015}, there exist non-zero polynomials $g_{1},g_{2}\in \bZ_p[X]$ such that $R = fg_{1} + f'g_{2}$. Since $p^j\nmid R$, the result follows.
\end{proof}

\begin{lemma} \label{lemA.3} Let $p$ be prime, and let $f(X) \in \bZ_p(X)$ be squarefree of degree $d \in \bN$. Define $Y_1,Y_2,\ldots$ as in \eqref{eqnA.2}. Let $h \in \bZ_{\ge 0}$, suppose $Y_{h+1}=\emptyset$, and let $s > 2h$ be an integer. Then
\[
\# \{ x \in [p^s]: f(x) \equiv 0 \mmod{p^{s}} \} \le dp^h.
\]
\end{lemma}

\begin{proof} Let $X,X_1,X_2,\ldots$ be as in \eqref{eqnA.1}. For $a \in X_s$, write $\del_a = \ord_p(f'(a))$, and note that $\del_a \le h$. By Hensel's lemma \cite[Lemma 3]{Woo1996}, if $a \in X_s$ then there exists $\tilde a \in X$ such that
\[
\tilde a \equiv a \mmod {p^{s-\del_a}}.
\]
Therefore
\[
\tilde a \equiv a \mmod {p^{s-h}},
\]
and so
\[
|X_s| \le p^h |X| \le dp^h.
\]
\end{proof}

\begin{lemma} \label{lemA.4} Let $f(X) \in \bZ[X]$ be squarefree of degree $d \in \bN$, and let $m \in \bN$. Then
\[
\# \{ x \in [m]: f(x) \equiv 0 \mmod m \}
\ll_f d^{\ome(m)}.
\]
\end{lemma}

\begin{proof} Let $C = |c\Del|$, where $c$ is the leading coefficient of $f$ and $\Del$ is the discriminant of $f$. Let $h \in \bN$ be minimal such that
$Y_h$ is empty for any prime $p \le C$. By Lemma \ref{lemA.2}, we know that $h\ll_{f}1$.
For $p > C$ and $s \in \bN$, Lemma \ref{lemA.1} yields
\[
\# \{ x \in [p^s]: f(x) \equiv 0 \mmod {p^s} \} \le d.
\]
By the Chinese remainder theorem, it remains to show that if $p \le C$ and $s \in \bN$ then
\[
\# \{ x \in [p^s]: f(x) \equiv 0 \mmod{ p^s} \} \le dp^{2h}.
\]
This is trivial if $s \le 2h$, and otherwise it follows from Lemma \ref{lemA.3}.
\end{proof}

Let $w$ and $W$ be as defined in \S\ref{sec4}. The following lemma, based on Rankin's trick, is analogous to \cite[Lemma A.3]{CLP2021}. 
\begin{lemma} \label{lemA.5} Let $P(X) \in \bZ[X]$ have degree $d \ge 2$. Then there are at most $O_P(10^w W M^{-1/2})$ integers $b \in [W]$ such that 
\[
(P'(b),W)_d > M,
\]
where $(P'(b),W)_d$ denotes the largest $m \in \bN$ for which $m^d \mid (P'(b),W)$.
\end{lemma}

\begin{proof} By Lemma \ref{lemA.4}, we have
\[
\# \{ z \in [m]: f(z) \equiv 0 \mmod m \} \ll_f (\deg f)^{\ome(m)} \ll_f m^\eps 
\]
for any squarefree polynomial $f(X) \in \bZ[X]$. Applying this to the squarefree kernel of $P'$, our count is at most a constant times
\begin{align*} \sum_{\substack{m > M \\ m \mid W}} \frac{W}{m} m^\eps &\ll 
\sum_{\substack{m > M \\ w\text{-smooth}}} Wm^{\eps-1} \sqrt{\frac m M}
\\ &\leqslant WM^{-1/2} \prod_{p\le w}\left( 1 + \frac1{1-p^{\eps-1/2}} \right) \ll 10^w W M^{-1/2}.
\end{align*}
\end{proof}

\providecommand{\bysame}{\leavevmode\hbox to3em{\hrulefill}\thinspace}

\end{document}